\documentclass[12pt]{article}
\usepackage[left=2.0cm,right=2.0cm,top=3.0cm,bottom=3.0cm]{geometry}
\usepackage{graphicx}
\usepackage{subcaption} 
\usepackage{amsmath, amsthm, amsfonts, amssymb, color,textcomp}
\usepackage{mathrsfs}
\usepackage{bm}
\usepackage[colorlinks,linkcolor=blue,anchorcolor=blue,citecolor=blue]{hyperref}

\newtheorem{thm}{Theorem}[section]
\theoremstyle{definition} 
\usepackage{float}
\newtheorem{example}{Example}

\usepackage{xcolor}
\newtheorem{lem}[thm]{Lemma}
\newtheorem{prp}[thm]{Proposition}

\newtheorem{rem}[thm]{Remark}

\definecolor{wco}{rgb}{0.5,0.2,0.3}
\numberwithin{equation}{section} \theoremstyle{remark}

\newcommand{\F}{\mathcal{F}}

\title{{\bf On the infinite time horizon approximation for  McKean-Vlasov SDEs driven by L\'evy processes with common noise}
	\footnote{Supported in
		part by  NNSFC (No.12101390, No.12426656). }}
\author{Ke Xu$^{a,b}$, Fen-Fen Yang$^{a,b}$, Chenggui Yuan$^{c}$ \\	
	{\small $^a$ Department of Mathematics, Shanghai University, Shanghai 200444, China}\\
	{\small $^b$ Newtouch Center for Mathematics, Shanghai University, Shanghai, 200444, China}\\
	{\small $^c$  Department of Mathematics, Swansea University, Bay campus, SA1 8EN, UK }\\
	{\small  xuke@shu.edu.cn; yangfenfen@shu.edu.cn; c.yuan@swansea.ac.uk}}	
\date{}
\begin{document}
	\allowdisplaybreaks
	\def\d{\text{{\rm{d}}}}
	\maketitle
	
	\begin{abstract}
		In this work, we establish the existence and uniqueness of solutions to McKean-Vlasov stochastic differential equations (SDEs) driven by L\'evy processes with common noise, by means of a contraction mapping principle in the space of probability measures. In addition, on the infinite time horizon we analyse the propagation of chaos for  McKean-Vlasov SDEs driven by L\'evy processes in the presence of common noise, and provide numerical simulations to validate the corresponding model.
	\end{abstract}
	\section{Introduction}
	Let $(\Omega, \mathcal{F}, \mathbb{P})$ be a complete probability space. Consider the $d$-dimensional stochastic process $X = (X_t)_{t \geq 0}$ satisfying the following conditional McKean-Vlasov stochastic differential equation (SDE) with jumps given by:
	\begin{equation}\label{eq:MKV-SDE}
		\begin{aligned}
			\mathrm{d}X(t) &= b\bigl(X(t), \mathcal{L}^1(X(t))\bigr)  \mathrm{d}t + \sigma\bigl(X(t), \mathcal{L}^1(X(t))\bigr)  \mathrm{d}W_t \\
			&\quad + c\bigl(X(t), \mathcal{L}^1(X(t))\bigr)  \mathrm{d}Z_t + \sigma^{0}\bigl(X(t), \mathcal{L}^1(X(t))\bigr)  \mathrm{d}W^{0}_t,
		\end{aligned}
	\end{equation}
	for $t \geq 0$, with initial value $X_0=x_0$. Here $\mathcal{L}^1(X(t))$ provides a version of the conditional law of $X(t)$ given $W^0$, $W = (W_t)_{t \geq 0}$ is a $d$-dimensional standard Brownian motion.
	Denote $\mathbb{R}_0^d := \mathbb{R}^d \setminus \{0\}$ and let $\mathcal{B}(\mathbb{R}_+ \times \mathbb{R}_0^d)$ be the Borel $\sigma$-algebra on $\mathbb{R}_+ \times \mathbb{R}_0^d$, and
	\[
	Z_t = \int_0^t \int_{\mathbb{R}_0^d} z  \widetilde{N}(\mathrm{d}s, \mathrm{d}z), \quad t \geq 0,
	\]
	where $\widetilde{N}(\mathrm{d}t, \mathrm{d}z) := N(\mathrm{d}t, \mathrm{d}z) - \nu(\mathrm{d}z) \mathrm{d}t$ is the compensated Poisson random measure associated with the Poisson random measure $N$ on $(\mathbb{R}_+ \times \mathbb{R}_0^d, \mathcal{B}(\mathbb{R}_+ \times \mathbb{R}_0^d))$ with intensity measure $\nu(\mathrm{d}z) \mathrm{d}s$. Explicitly, for $t > 0$ and $A \in \mathcal{B}(\mathbb{R}_0^d)$,
	\[
	N([0, t] \times A) := \sum_{s \in (0,t]} \mathbf{1}_{\{ \Delta Z_s \in A \}},
	\]
	with the jump size of $Z$ at time $s$ defined by
	\begin{equation*}
		\Delta Z_s := 
		\begin{cases} 
			Z_s - Z_{s-}, & s > 0, \\
			0, & s = 0,
		\end{cases}
		\quad \text{where} \quad Z_{s-} := \lim_{u \uparrow s} Z_u.
	\end{equation*}
	Moreover $Z = (Z_t)_{t \geq 0}$ is a $d$-dimensional centered pure jump L\'evy process independent of $W$, whose L\'evy measure $\nu$ satisfies 
	\[
	\int_{\mathbb{R}^d} (1 \wedge |z|^2)  \nu(\mathrm{d}z) < +\infty.
	\]
	Let $|\cdot|$ denote the Euclidean norm on $\mathbb{R}^d$, and let $(\mathcal{F}^0_t)_{t \ge 0}$, $(\mathcal{F}^1_t)_{t \ge 0}$ be two complete filtered probability spaces satisfying the usual condition.
	Let $(W_t^0)_{t \ge 0}$ and $(W_t)_{t \ge 0}$ be independent $d$-dimensional Brownian motions defined on $(\Omega^0, \mathcal{F}^0, \mathbb{P}^0)$ and $(\Omega^1, \mathcal{F}^1, \mathbb{P}^1)$, respectively.
	We construct the product probability space $(\Omega, \mathcal{F}, \mathbb{P})$ by setting
	\[
	\Omega := \Omega^0 \times \Omega^1, \quad 
	(\mathcal{F}, \mathbb{P}) := \text{completion of } (\mathcal{F}^0 \otimes \mathcal{F}^1,\, \mathbb{P}^0 \otimes \mathbb{P}^1),
	\]
	and equipping it with the filtration $(\mathcal{F}_t)_{t \ge 0}$ obtained as the complete, right-continuous augmentation of $(\mathcal{F}^0_t \otimes \mathcal{F}^1_t)_{t \ge 0}$.  
	An element $\omega \in \Omega$ means $\omega = (\omega^0, \omega^1)$ with $\omega^0 \in \Omega^0$ and $\omega^1 \in \Omega^1$. Expectations with respect to $\mathbb{P}^0$ and $\mathbb{P}^1$ are denoted by $\mathbb{E}^0$ and $\mathbb{E}^1$, respectively.  
	For $i \in \{1,2\}$, let $(\Omega_i^1, \mathcal{F}_i^1, \mathbb{P}_i^1)$ be independent copies of $(\Omega^1, \mathcal{F}^1, \mathbb{P}^1)$, with $\mathbb{E}^1_i$ denoting expectation under $\mathbb{P}_i^1$.  
	We write $\Omega_i := \Omega^0 \times \Omega_i^1$ for the corresponding copies of $\Omega$.
	Denote by $\mathcal{P}(\mathbb{R}^d)$ the set of all Borel probability measures on $\mathbb{R}^d$ endowed with the topology of weak convergence.  
	For $k>0$, define
	\[
	\mathcal{P}_k(\mathbb{R}^d) := \left\{ \mu \in \mathcal{P}(\mathbb{R}^d) : \int_{\mathbb{R}^d} |x|^k\, \mu(\mathrm{d}x) < \infty \right\}.
	\]
	For $\mu, \nu \in \mathcal{P}_k(\mathbb{R}^d)$, the $k$-Wasserstein distance is given by
	\[
	\mathcal{W}_k(\mu, \nu) := \inf_{\pi \in \Pi(\mu, \nu)} \left( \int_{\mathbb{R}^d \times\mathbb{R}^d} |x - y|^k\, \pi(\mathrm{d}x, \mathrm{d}y) \right)^{\frac{1}{k \vee 1}},
	\]
	where $\Pi(\mu,\nu)$ denotes the set of all couplings of $\mu$ and $\nu$.  
	It is known that $(\mathcal{P}_k(\mathbb{R}^d), \mathcal{W}_k)$ forms a Polish metric space. We denote $\delta_0$ the Dirac measure concentrated at the origin in $\mathbb{R}^d$. 
	Let $\xi$ be a random variable on $\Omega$ measurable with respect to $\mathcal{F}^0 \otimes \mathcal{F}^1$.  
	For any fixed $\omega^0 \in \Omega^0$, the mapping
	\[
	\Omega^1 \ni \omega^1 \longmapsto \xi(\omega^0, \omega^1)
	\]
	defines a random variable on $(\Omega^1, \mathcal{F}^1, \mathbb{P}^1)$, whose distribution is denoted by $\mathcal{L}(\xi(\omega^0, \cdot))$.  
	Given a random variable $\xi$ on $(\Omega, \mathcal{F}, \mathbb{P})$, we define
	\[
	\mathcal{L}^1(\xi) := \mathcal{L}(\xi(\omega^0, \cdot)),
	\]
	which represents the conditional distribution of $\xi$ given $\mathcal{F}^0$.  
	In particular, if the initial condition $X(0)$ in \eqref{eq:MKV-SDE} is assumed to be defined on $(\Omega^1, \mathcal{F}^1, \mathbb{P}^1)$, then $\mathcal{L}^1(X(t))$ provides a version of the conditional law of $X(t)$ given $W^0$.  
	Henceforth, we shall restrict our attention to initial conditions living on $(\Omega^1, \mathcal{F}^1, \mathbb{P}^1)$. 
	
	The existence and uniqueness theory for strong solutions of McKean-Vlasov SDEs under Lipschitz and linear growth conditions (with respect to both the state variable and the measure) is well-established (see, e.g., \cite{Meleard1996}, \cite{Sznitman1991}). Wang~\cite{Wang2018} establishes well-posedness via the Picard iteration method, resolving the pathwise intractability in the case of SDEs driven by Brownian motion. In the context of delay McKean-Vlasov SDEs with super-linear drift, Bao et al.~\cite{Bao20242437} propose a Milstein-based antithetic MLMC framework attaining optimal complexity without L\'evy area simulation under Brownian noise. Liu et al.~\cite{LWZ2021} establish concentration inequalities, exponential convergence in the Wasserstein metric $\mathcal{W}_1$, and uniform-in-time propagation of chaos for McKean-Vlasov systems via reflection coupling and new cost functions, with applications to multi-well confinement and bounded interaction potentials.
	
	Compared with McKean-Vlasov SDEs driven by Brownian motion, the McKean-Vlasov SDEs driven by L\'evy processes have been much less studied. Bao and Wang~\cite{Bao2025} establish existence, uniqueness, and multiplicity of stationary distributions for  McKean-Vlasov SDEs driven by L\'evy processes. Song et al.~\cite{Liu2023} develop large and moderate deviation principles, along with explicit rate functions, for  McKean-Vlasov SDEs driven by L\'evy processes using a weak-convergence approach. In \cite{Jourdain2008}, Jourdain et al. investigate the existence, uniqueness, and particle approximations for McKean-Vlasov SDEs driven by L\'evy processes. More recently, Mehri et al.~\cite{Mehri2020} and Neelima et al.~\cite{Neelima2020} address the well-posedness and propagation of chaos results for McKean-Vlasov SDEs with delay driven by L\'evy noise. Bao et al.~\cite{Bao20254} establish strong well-posedness of McKean-Vlasov SDEs driven by L\'evy processes under weak monotonicity and coercivity, prove both weak and strong propagation of chaos, and extend the results to systems with common noise.
	
	The McKean-Vlasov SDEs driven by Brownian motion with common noise have been applied to various fields including finance, physics, biology, mean-field games, and machine learning. The fundamental theory of the McKean-Vlasov SDEs with common noise is well established. One can refer to Carmona and Delarue~\cite{Carmona2018}, Kumar et al.~\cite{Kumar2020} for the well-posedness and conditional propagation of chaos, Hammersley et al.~\cite{Hammersley2021} for the weak well-posedness, Maillet~\cite{Maillet2023}, Bao and Wang~\cite{Bao2024} for the long-time behavior. Yuan et al.~\cite{Chen2025105993} design delay feedback control strategies for stabilizing McKean-Vlasov SDEs with common noise, proving the global solution existence and uniqueness.
	Shao and Wei~\cite{Shao2021} study the well-posedness and conditional propagation of chaos for McKean-Vlasov SDEs with common noise, establishing explicit convergence rates.
	
In applications, random jump shocks naturally arise in a variety of systems. 
For instance, an epidemic outbreak occurs when a contagious disease spreads rapidly through a population due to increased transmission rates. 
Super-spreader events may cause sudden surges in infection numbers, which can be modelled by a jump component in a diffusion-type dynamics for the infected population. 
When global mobility is high, localized outbreaks can quickly propagate across regions, resulting in a pandemic, and the epidemic dynamics are then subject to system-wide sources of randomness (e.g., environmental factors, global travel) that impact all communities. Similar large-scale effects appear in economics and finance. 
In systemic risk modelling, the balance sheet of each bank is influenced not only by its idiosyncratic shocks but also by macroeconomic factors and market-wide events. 
This can be described by McKean-Vlasov SDEs with common noise and jumps, where the mean-field term captures peer effects and the jump component models sudden losses or liquidity dry-ups; see the interbank lending model studied later. 
Bistable dynamics with random environment, as in a Ginzburg-Landau type equation, provide a prototype for metastable transitions under heavy-tailed shocks, and will serve as our first test model. 
In conditional mean-field control problems, agents interact through conditional expectations of their states while being exposed to common noise representing shared information or policy shocks. 
Linear-quadratic structures then lead to tractable feedback controls and motivate the LQ test models considered in Section~5.
These examples motivate the study of McKean-Vlasov SDEs driven by L\'evy processes with common noise; see, for instance, \cite{BWWY2025,BWWY2024}. 
In recent years, significant advances have been achieved for McKean-Vlasov SDEs driven by L\'evy processes under common noise, covering topics such as well-posedness, deep learning approaches, optimal stopping problems, optimal and impulse control, as well as stochastic maximum principles; see, among others, \cite{AR2025,AO2024,AO2023,BWWY2025,BWWY2024,Do2024,Erny2022,Huang2025,HR2024,McKean1966} for further references.

	This paper investigates the infinite time horizon approximation for McKean-Vlasov SDEs driven by L\'evy processes with common noise. The primary contribution of this study lies in addressing a research gap by extending previous work to McKean-Vlasov SDE driven by L\'evy processes, where the drift and diffusion coefficients are not globally Lipschitz continuous. Tran et al.~\cite{Tran2025} propose a tamed-adaptive Euler-Maruyama scheme for McKean-Vlasov SDEs driven by L\'evy processes with superlinear coefficients, proving strong convergence on finite and infinite horizons. In particular, this study builds upon the results of Tran et al.~\cite{Tran2025} by introducing the effect of common noise. This factor has not been examined in the context of such SDEs. The incorporation of common noise enhances the model, providing a more realistic framework for situations where noise simultaneously impacts all components, a consideration that has often been neglected in prior research.
	
	The main contributions of this article are reflected in the following three aspects. 
	
	\begin{enumerate}
		\item Establishing existence, uniqueness, and moment estimates for solutions under non-globally Lipschitz coefficients (assumptions \textbf{(A1)-(A7)}), generalizing results in \cite{ {Kumar2020},{Neelima2020},{Shao2024}}.
		
		\item We extend the results of Tran et al.~\cite{Tran2025} to establish the boundedness of high-order moments of the solutions and demonstrate the propagation of chaos in the presence of common noise.
		
		\item We apply the theoretical proofs presented in this paper to practical examples and provide corresponding numerical analysis and results.
	\end{enumerate}
	
	The remainder of the paper is organized as follows. In Section 2, we establish the well-posedness of solutions to McKean-Vlasov SDE driven by L\'evy process with common noise, including existence, uniqueness, and second-order moment estimates under non-globally Lipschitz coefficients. In Section 3, we investigate the propagation of chaos for the associated particle systems. In Section 4, we present the corresponding numerical analysis of the McKean-Vlasov SDE driven by L\'evy process with common noise. In Section 5, we provide specific examples to conduct numerical simulations.
	\section{Well-posedness of Solutions}
	To establish the well-posedness of \eqref{eq:MKV-SDE}, we assume that the drift, diffusion and jump coefficients \( b \), \( \sigma \), \(\sigma^{0}\), \( c \) and the L\'evy measure \( \nu \) satisfy the following conditions. 
	\vspace*{0.5cm}
	\\
	\textbf{(A1)} There exists a positive constant \( L \) such that
	\[
	2 \langle x, b(x, \mu) \rangle + |\sigma(x, \mu)|^2 + |\sigma^{0}(x, \mu)|^2 + |c(x, \mu)|^2 \int_{\mathbb{R}^d_{0}} |z|^2 \, \nu({{\rm{d}}}z) \leq L(1 + |x|^2 + \mathcal{W}_2^2(\mu, \delta_0)
	),
	\]
	for any \( x\in \mathbb{R}^d \) and \( \mu\in \mathcal{P}_2(\mathbb{R}^d) \).\\
	\textbf{(A2)} There exist constants \( L_1 \in \mathbb{R} \) and \( L_2 \geq 0 \) such that 
	\begin{align*}
		&2 \langle x - \overline{x}, b(x, \mu) - b(\overline{x}, \overline{\mu}) \rangle + |\sigma(x, \mu) - \sigma(\overline{x}, \overline{\mu})|^2 +  |\sigma^{0}(x, \mu) - \sigma^{0}(\overline{x}, \overline{\mu})|^2 \\
		&+ |c(x, \mu) - c(\overline{x}, \overline{\mu})|^2\int_{\mathbb{R}^d_{0}} |z|^2 \, \nu({{\rm{d}}}z)
		\leq L_1 |x -\overline{x}|^2 + L_2 \mathcal{W}_2^2(\mu, \overline{\mu}),
	\end{align*}for any \( x, \overline{x} \in \mathbb{R}^d \) and \( \mu, \overline{\mu} \in \mathcal{P}_2(\mathbb{R}^d) \).\\
	\textbf{(A3)} There exist constants \( \tilde{L} > 0 \) and \( \ell > 0 \) such that 
	\[
	|b(x, \mu) - b(\overline{x}, \overline{\mu})| \leq \tilde{L}(1 + |x|^\ell + |\overline{x}|^\ell) \left( |x -\overline{x}| + \mathcal{W}_2(\mu, \overline{\mu}) \right),
	\]for any \( x,\overline{x}\in \mathbb{R}^d \) and \( \mu, \overline{\mu} \in \mathcal{P}_2(\mathbb{R}^d) \).\\
	\textbf{(A4)} There exists an even integer \( p_0 \in [2, +\infty) \) such that
	\[
	\int_{|z| > 1} |z|^{2p_0}\, \nu({{\rm{d}}}z) < \infty \quad \text{and} \quad \int_{0 < |z| \leq 1} |z| \, \nu({{\rm{d}}}z) < \infty.
	\]\\
	\textbf{(A5)} There exists a positive constant \( L_3 \) such that 
	\[
	|c(x, \mu)| \leq L_3 \left( 1 + |x| + \mathcal{W}_2(\mu, \delta_0) \right),
	\]for any \( x \in \mathbb{R}^d \) and \( \mu \in \mathcal{P}_2(\mathbb{R}^d) \).\\
	\textbf{(A6)} For the even integer \( p_0 \in [2, +\infty) \) given in \textbf{(A4)} and the positive constant \( L_3 \) given in \textbf{(A5)}, there exist constants \( \gamma_1 \in \mathbb{R} \), \( \gamma_2 \geq 0 \), and \( \eta \geq 0 \) such that
	\begin{align*}
		&\langle x, b(x, \mu) \rangle + \frac{p_0 - 1}{2} (|\sigma(x, \mu)|^2 + |\sigma^{0}(x, \mu)|^2 )+ \frac{1}{2L_3} |c(x, \mu)|^2 
		\int_{\mathbb{R}^d_{0}}  |z| \left(\left( 1 + L_3 |z| \right)^{p_0 - 1} - 1 \right) \nu({{\rm{d}}}z) \\
		&\leq \gamma_1 |x|^2 + \gamma_2 \mathcal{W}_2^2(\mu, \delta_0) + \eta,
	\end{align*} for any \( x \in \mathbb{R}^d \) and \( \mu \in \mathcal{P}_2(\mathbb{R}^d) \).\\
		\textbf{(A7)} There exist constants \( \tilde{L}_1 \in \mathbb{R} \) and \( \tilde{L}_2 \geq 0 \) such that 
	\begin{align*}
		&|\sigma(x, \mu) - \sigma(\overline{x}, \overline{\mu})|^2 +  |\sigma^{0}(x, \mu) - \sigma^{0}(\overline{x}, \overline{\mu})|^2 + |c(x, \mu) - c(\overline{x}, \overline{\mu})|^2\int_{\mathbb{R}^d_{0}} |z|^2 \, \nu({{\rm{d}}}z)\\
		&\leq \tilde{L}_1 |x -\overline{x}|^2 + \tilde{L}_2 \mathcal{W}_2^2(\mu, \overline{\mu}),
	\end{align*}for any \( x, \overline{x} \in \mathbb{R}^d \) and \( \mu, \overline{\mu} \in \mathcal{P}_2(\mathbb{R}^d) \).
	\begin{rem}\label{rem}
		Similar with Tran et al.~\cite{Tran2025}. It follows from Assumption \textup{\textbf{(A6)}} that, for the even integer \( p_0 \in [2, +\infty) \) given in \textbf{(A4)} and the positive constant \( L_3 \) given in \textbf{(A5)}, there exist constants \( \gamma_1 \in \mathbb{R} \), \( \gamma_2 \geq 0 \), and \( \eta \geq 0 \) such that
		\begin{align*}
			&\langle x, b(x, \mu) \rangle + \frac{p - 1}{2} (|\sigma(x, \mu)|^2 + |\sigma^{0}(x, \mu)|^2 )+ \frac{1}{2L_3} |c(x, \mu)|^2 
			\int_{\mathbb{R}^d_{0}}  |z| \left(\left( 1 + L_3 |z| \right)^{p - 1} - 1 \right) \nu({{\rm{d}}}z) \\
			&\leq \gamma_1 |x|^2 + \gamma_2 \mathcal{W}_2^2(\mu, \delta_0) + \eta,
		\end{align*}
		for any \( p \in [2, p_0] \), \( x \in \mathbb{R}^d \) and \( \mu \in \mathcal{P}_2(\mathbb{R}^d) \).
	\end{rem}
	\begin{example}\label{vv}
		To illustrate that the aforementioned hypotheses are valid, we consider the following example: a variant of the Ginzburg-Landau equation. Specifically, consider		
		\[
		\begin{aligned}
			&{\rm d}X_t = 0.1\big(X_t - X_t^3 + m(\mu_t)\big)\,{\rm d}t + 0.1X_t\,{\rm d}W_t + 0.05X_t\,{\rm d}W^0_t + \sin(X_{t-})\,{\rm d}Z_t,\\
			& i.e. \quad b(x,\mu) = 0.1\big(x - x^3 + m(\mu)\big), \quad 
			\sigma(x,\mu) = 0.1x,\quad \sigma^{0}(x,\mu) = 0.05x, \quad 
			c(x,\mu) = \sin(x),
		\end{aligned}
		\]
		where $\mu_t = \mathcal L^1(X_t)$ denotes the conditional distribution of $X_t$, $m(\mu_t) = \int y\,\mu_t({\rm d}y)$ is its first moment, $W$ is a Brownian motion, $W^0$ a common Brownian motion, and $Z$ a pure-jump L\'evy process with L\'evy measure $\nu({\rm d}z)$. For the jump part we choose a tempered stable L\'evy measure on $\mathbb R_0$:
		\[\nu(\mathrm{d}z)=C_{\alpha,\lambda}\,|z|^{-1-\alpha}e^{-\lambda|z|}\,\mathrm{d}z,\qquad \alpha\in(0,1),\;\lambda>0,
		\]
    	where $C_{\alpha,\lambda}$ is determined by the normalization condition ensuring that $\nu$ is a valid L\'evy measure, i.e.
    	\[
    	C_{\alpha,\lambda} = \frac{\lambda^{\alpha}}{2\Gamma(-\alpha)},
    	\]
    	with $\Gamma(\cdot)$ denoting the Gamma function.
		Note that  $\Gamma(s)=\int_{0}^{\infty} t^{s-1} e^{-t}\,\mathrm{d}t$ is the Gamma function. By setting $s=q-\alpha$, we have
		\[
		\int_{\mathbb R_0}|z|^q\,\nu(\mathrm{d}z)
		= C_{\alpha,\lambda}\int_{\mathbb R_0} |z|^{q-1-\alpha} e^{-\lambda|z|}\,\mathrm{d}z
		=2C_{\alpha,\lambda}\,\lambda^{\alpha-q}\,\Gamma(q-\alpha),\qquad q>\alpha.
		\]
		For $q=2$, we obtain
		\[
		\int_{\mathbb R_0}|z|^2\,\nu(\mathrm{d}z)
		=2C_{\alpha,\lambda}\,\lambda^{\alpha-2}\,\Gamma(2-\alpha), \qquad 0<\alpha<1.
		\]					
		Next we verify that the above equation satisfies assumptions \textup{\textbf{(A1)-(A7)}}.
		\\
		\textbf{(1)}
		Since $b(x,\mu)=0.1(x-x^3+m(\mu))$, we have
		\begin{align*}
			2 \langle x, b(x, \mu) \rangle &=2x\cdot0.1\big(x-x^3+m(\mu)\big)=0.2\big(x^2-x^4+xm(\mu)\big).
		\end{align*}
	   Moreover,
		\[|\sigma(x,\mu)|^2=(0.1x)^2=0.01x^2,\qquad |\sigma^0(x,\mu)|^2=(0.05x)^2=0.0025x^2.
		\]		
		Also we have $|c(x,\mu)|^2\int_{\mathbb R_0}|z|^2\,\nu(\mathrm{d}z)=\sin^2(x)\cdot 2C_{\alpha,\lambda}\,\lambda^{\alpha-2}\,\Gamma(2-\alpha)\le 2C_{\alpha,\lambda}\,\lambda^{\alpha-2}\,\Gamma(2-\alpha)$ and $m(\mu)^2\le\mathcal W_2^2(\mu,\delta_0).$ Thus, we have
		\begin{align*}
			&2 \langle x, b(x, \mu) \rangle+|\sigma(x,\mu)|^2+|\sigma^0(x,\mu)|^2+|c(x,\mu)|^2\int_{\mathbb R_0}|z|^2\,\nu(\mathrm{d}z)\\
			&\quad=0.2(x^2-x^4+xm(\mu))+0.0125x^2+\sin^2(x)\cdot 2C_{\alpha,\lambda}\,\lambda^{\alpha-2}\,\Gamma(2-\alpha)\\
			&\quad\le 0.2x^2+0.0125x^2+0.2|x||m(\mu)|+2C_{\alpha,\lambda}\,\lambda^{\alpha-2}\,\Gamma(2-\alpha)\\
			&\quad \le (0.2125+0.1)x^2+0.1\mathcal W_2^2(\mu,\delta_0)+2C_{\alpha,\lambda}\,\lambda^{\alpha-2}\,\Gamma(2-\alpha)\\
			&\quad =0.3125 x^2+0.1\mathcal W_2^2(\mu,\delta_0)+2C_{\alpha,\lambda}\,\lambda^{\alpha-2}\,\Gamma(2-\alpha).
		\end{align*}
		Letting $L=\max\{0.3125,0.1,2C_{\alpha,\lambda}\,\lambda^{\alpha-2}\,\Gamma(2-\alpha)\}+1$, we show that the Assumption \textbf{(A1)} holds.\\
		\textbf{(2)}
		Note that
		\[|\sigma(x)-\sigma(\overline x)|^2=(0.1)^2|x-\overline x|^2=0.01|x-\overline x|^2,\]
		\[|\sigma^0(x)-\sigma^0(\overline x)|^2=(0.05)^2|x-\overline x|^2=0.0025|x-\overline x|^2.
		\]
		For the jump coefficient, use $|\sin x-\sin y|\le|x-y|$, so
		
		\[|c(x)-c(\overline x)|^2\int_{\mathbb R_0}|z|^2\,\nu(\mathrm{d}z)\le 2C_{\alpha,\lambda}\,\lambda^{\alpha-2}\,\Gamma(2-\alpha)|x-\overline x|^2.
		\]
		Hence,
		\begin{align*}
			&2 \langle x - \overline{x}, b(x, \mu) - b(\overline{x}, \overline{\mu}) \rangle\\
			&\quad=2(x-\overline x)\cdot0.1\big((x-x^3)-(\overline x-\overline x^3)+m(\mu)-m(\overline\mu)\big)\\
			&\quad=0.2 (x-\overline x)\big((x-x^3)-(\overline x-\overline x^3)\big)+0.2(x-\overline x)(m(\mu)-m(\overline\mu)).
		\end{align*}
		Using the algebraic identity	
		\((x-x^3)-(\overline x-\overline x^3)=(x-\overline x)\big(1-(x^2+x\overline x+\overline x^2)\big),
		\)
		we have
		\[0.2 (x-\overline x)\big((x-x^3)-(\overline x-\overline x^3)\big)=0.2 |x-\overline x|^2\big(1-(x^2+x\overline x+\overline x^2)\big)\le 0.2|x-\overline x|^2.
		\]
		Moreover,
		\[0.2(x-\overline x)(m(\mu)-m(\overline\mu))\le 0.1|x-\overline x|^2+0.1|m(\mu)-m(\overline\mu)|^2.
		\]		
		Combining the above bounds gives
		\begin{align*}
			&2 \langle x - \overline{x}, b(x, \mu) - b(\overline{x}, \overline{\mu}) \rangle+|\sigma(x,\mu)-\sigma(\overline x,\overline\mu)|^2+|\sigma^0(x,\mu)-\sigma^0(\overline x,\overline\mu)|^2\\
			&\le \big(0.2+0.1+0.01+0.0025+2C_{\alpha,\lambda}\,\lambda^{\alpha-2}\,\Gamma(2-\alpha)\big)|x-\overline x|^2+0.1|m(\mu)-m(\overline\mu)|^2\\
			&=\Big(0.3125+2C_{\alpha,\lambda}\,\lambda^{\alpha-2}\,\Gamma(2-\alpha)\Big)|x-\overline x|^2+0.1|m(\mu)-m(\overline\mu)|^2.
		\end{align*}
		Since $|m(\mu)-m(\overline\mu)|\le\mathcal W_2(\mu,\overline\mu)$, we may choose
		
		\[L_1:=0.3125+2C_{\alpha,\lambda}\,\lambda^{\alpha-2}\,\Gamma(2-\alpha),\qquad L_2:=0.1.
		\]		
		Hence Assumption \textbf{(A2)} holds with these constants.\\
		\textbf{(3)}
		Due to
		\begin{align*}
			|b(x,\mu)-b(\overline x,\overline\mu)|&=0.1\big|(x-x^3)-(\overline x-\overline x^3)+m(\mu)-m(\overline\mu)\big|\\
			&\le 0.1\big| (x-x^3)-(\overline x-\overline x^3)\big|+0.1|m(\mu)-m(\overline\mu)|.
		\end{align*}
		Applying the mean-value theorem: there exists $\xi$ between $x$ and $\overline x$ such that		
		\[|(x-x^3)-(\overline x-\overline x^3)|=|x-\overline x| |1-3\xi^2|\le |x-\overline x|(1+3\xi^2) \le |x-\overline x|(1+3(|x|^2+|\overline x|^2)).
		\]		
		Thus
		\[|b(x,\mu)-b(\overline x,\overline\mu)|\le 0.1\big(1+3(|x|^2+|\overline x|^2)\big)|x-\overline x|+0.1\mathcal W_2(\mu,\overline\mu).
		\]
	This implies that Assumption \textbf{(A3)} holds with $\ell=2$.\\
		\textbf{(4)}
		Note that  $\Gamma(s,\lambda)=\int_{\lambda}^{\infty} t^{s-1} e^{-t}\,\mathrm{d}t$ is the upper incomplete Gamma function. 	When $2p_0>\alpha$, we have
		\begin{align*}
			\int_{|z|>1}|z|^{2p_0}\,\nu(\mathrm{d}z)
			&=2C_{\alpha,\lambda}\int_{1}^{\infty} r^{2p_0-1-\alpha} e^{-\lambda r}\,\mathrm{d}r\\
			&=2C_{\alpha,\lambda}\,\lambda^{\alpha-2p_0}\,\Gamma(2p_0-\alpha,\lambda).
		\end{align*}
		Similarly, due to $\gamma(s,\lambda)=\int_{0}^{\lambda} t^{s-1}e^{-t}\,\mathrm{d}t$ is the lower incomplete Gamma function. When $\alpha<1$, we have
		\begin{align*}
			\int_{0<|z|\le 1}|z|\,\nu(\mathrm{d}z)
			&=2C_{\alpha,\lambda}\int_{0}^{1} r^{-\alpha} e^{-\lambda r}\,\mathrm{d}r\\
			&=2C_{\alpha,\lambda}\,\lambda^{\alpha-1}\,\gamma(1-\alpha,\lambda).
		\end{align*}
		Thus Assumption \textbf{(A4)} holds.\\
		\textbf{(5)}
		Since $|\sin x|\le 1\le 1+|x|+\mathcal W_2(\mu,\delta_0)$ for all $x,\mu$, it's easy to see that Assumption \textbf{(A5)} holds.\\
		\textbf{(6)}
		Since $b(x,\mu)=0.1(x-x^3+m(\mu))$,
		\[\langle x, b(x, \mu) \rangle=0.1\big(x^2-x^4+xm(\mu)\big)\le 0.1x^2+0.1|x||m(\mu)|\le 0.1x^2+0.05x^2+0.05\mathcal W_2^2(\mu,\delta_0).\]
		Next,
		\begin{align*}
			\frac{p_0-1}{2}(|\sigma(x,\mu)|^2+|\sigma^0(x,\mu)|^2)&=\frac{p_0-1}{2}\cdot0.0125 x^2.
		\end{align*}
		Using $|c(x,\mu)|^2=\sin^2(x)\le 1$ and $L_3=1$, we have 
		\begin{align*}
			&\langle x,b(x,\mu)\rangle+\frac{p_0-1}{2}(|\sigma(x,\mu)|^2+|\sigma^0(x,\mu)|^2)+\frac{1}{2L_3}|c(x,\mu)|^2\int_{\mathbb R_0} |z|\Big((1+L_3|z|)^{p_0-1}-1\Big)\nu(\mathrm{d}z)\\
			&\quad \leq \Big(0.15+\frac{p_0-1}{2}\cdot0.0125\Big)x^2+0.05\mathcal W_2^2(\mu,\delta_0)+\tfrac12 \int_{\mathbb R_0}|z|\Big((1+|z|)^{p_0-1}-1\Big)\cdot C_{\alpha,\lambda}\,|z|^{-1-\alpha}e^{-\lambda|z|}\,\mathrm{d}z.
		\end{align*}
		Let
		\[
		I:=\tfrac12\int_{\mathbb R_0}|z|\Big((1+|z|)^{p_0-1}-1\Big)\nu(\mathrm{d}z).
		\]
		Setting $t=\lambda |z|$, we obtain
		\[
		\begin{aligned}
			I
			&= C_{\alpha,\lambda}\int_{0}^{\infty}\Big(\frac{t}{\lambda}\Big)^{-\alpha}
			\Big(\bigl(1+\tfrac{t}{\lambda}\bigr)^{p_0-1}-1\Big)e^{-t}\,\frac{\mathrm{d}t}{\lambda}\\[6pt]
			&= C_{\alpha,\lambda}\,\lambda^{\alpha-1}\int_{0}^{\infty} t^{-\alpha}
			\Big(\bigl(1+\tfrac{t}{\lambda}\bigr)^{p_0-1}-1\Big)e^{-t}\,\mathrm{d}t.
		\end{aligned}
		\]			
		As $t\to 0$, the term $\bigl(1+\tfrac{t}{\lambda}\bigr)^{p_0-1}-1 \sim (p_0-1)\tfrac{t}{\lambda}$, so it is integrable if $\alpha<2$. 
		As $t\to\infty$, exponential damping $e^{-t}$ ensures integrability for all $p_0$.	
		Thus the integral $I$ is finite for all $p_0\in\mathbb R$ and $\alpha\in(0,1)$.
		Hence we may choose	
		\[\gamma_1:=0.15+\frac{p_0-1}{2}\cdot0.0125,\qquad \gamma_2:=0.05,\qquad \eta:=\tfrac12 \int_{\mathbb R_0}|z|\Big((1+|z|)^{p_0-1}-1\Big)\nu(\mathrm{d}z).
		\]			
		This proves Assumption \textbf{(A6)}.\\
		\textbf{(7)} Note that
		\begin{align*}
			&|\sigma(x)-\sigma(\overline x)|^2+|\sigma^0(x)-\sigma^0(\overline x)|^2+|c(x)-c(\overline x)|^2\int_{\mathbb R_0}|z|^2\,\nu(\mathrm{d}z)\\
			&\quad=0.01|x-\overline x|^2+0.0025|x-\overline x|^2+2C_{\alpha,\lambda}\,\lambda^{\alpha-2}\,\Gamma(2-\alpha)|\sin x-\sin\overline x|^2\\
			&\quad\le (0.0125+2C_{\alpha,\lambda}\,\lambda^{\alpha-2}\,\Gamma(2-\alpha))|x-\overline x|^2.
		\end{align*}
		Hence choose $\tilde L_1:=0.0125+2C_{\alpha,\lambda}\,\lambda^{\alpha-2}\,\Gamma(2-\alpha)$ and $\tilde L_2:=0$. This verifies Assumption \textbf{(A7)}.		
	\end{example}
	\begin{rem}
		We point out that assumptions \textbf{(A1)}, \textbf{(A2)}, \textbf{(A4)}, \textbf{(A6)} and \textbf{(A7)} does \textbf{not} hold for an isotropic $\alpha$-stable L\'evy measure with index $\alpha \in (0,2)$, and the measure $\nu(\mathrm{d}z)=c_{d,\alpha}\,|z|^{-d-\alpha}\,\mathrm{d}z.$
		In fact
		\[
		\int_{|z|>1}|z|^2\,\nu(\mathrm{d}z)
		= c_{d,\alpha}\int_{|z|>1}|z|^{2-d-\alpha}\,\mathrm{d}z.
		\]
		Using spherical coordinates,
		\[
		\int_{|z|>1}|z|^{2-d-\alpha}\,\mathrm{d}z
		= \int_{1}^{\infty}\int_{S^{d-1}} r^{2-d-\alpha}\,r^{d-1}\,\mathrm{d}\sigma(\theta)\,\mathrm{d}r
		= S_{d-1}\int_{1}^{\infty} r^{1-\alpha}\,\mathrm{d}r,
		\]
		where $S_{d-1}=2\pi^{d/2}/\Gamma(d/2)$ is the surface area of the $(d-1)$-sphere.
		
		The radial integral is
		\[
		\int_{1}^{\infty} r^{1-\alpha}\,\mathrm{d}r
		= \frac{1}{2-\alpha}\left(r^{2-\alpha}\right)\Big|_{1}^{\infty}.
		\]
		For $\alpha\in(0,2)$ we have $2-\alpha>0$, hence the limit as $r\to\infty$ diverges. Thus
		\[
		\int_{|z|>1}|z|^2\,\nu(\mathrm{d}z)=\infty.
		\]
		Therefore the isotropic $\alpha$-stable L\'evy measure does not admit a finite second moment.
	\end{rem}

	\begin{thm}\label{00}
		Let assumptions \textup{\textbf{(A1)-(A7)}} hold. Then, for any random distribution $X^\mu_0 \in \mathcal{F}_0$ with $\mathbb{E}|X^\mu_0|^2<\infty$, there exists a unique c$\grave{a}$dl$\grave{a}$g process X taking values in $\mathbb{R}^d$ satisfying the McKean-Vlasov SDE \eqref{eq:MKV-SDE} with
		initial distribution \(X^\mu_0\) such that 
		\[\sup_{t \in [0,T]} \mathbb{E} [|X_t|^2] \leq K,\]
		where $K := K(\mathbb{E}|X^\mu_0|^2, d, L, L_1, \gamma_1, \gamma_2, T)$ is a positive constant.
	\end{thm}
	
	\begin{proof}
		To facilitate readability, we divide the proof into the following steps.
		\begin{enumerate}
			\item[${\bf (1)}$] \textbf{Constructing Metric Spaces}
		\end{enumerate}
		Let \(\mathcal{D}([0,T];\mathbb{R}^d)\) denote the space of c$\grave{a}$dl$\grave{a}$g functions on \([0,T]\), which equips with the uniform metric
		\[
		\rho_{\infty}(\xi,\tilde{\xi}) := \sup_{t\in[0,T]}\left\{|\xi_{t}-\tilde{\xi}_{t}| \right\}, \quad \xi,\tilde{\xi} \in \mathcal{D}([0,T];\mathbb{R}^{d}).
		\]
		Then, the space \(\mathcal{D}([0,T];\mathbb{R}^{d})\) endowed with the distance \(\rho_{\infty}\) is a complete space.
		
		For any given \(T>0\) and \(p\geq 1\), let \(\mathscr{Y}_{T}(p)\) be a space consisting of all processes \(\eta:[0,T]\times\Omega\to\mathbb{R}^{d}\) satisfying	
		\begin{itemize}
			\item[(i)] For almost all \(\omega\in\Omega\), \(\eta_{.}(\omega):[0,T]\to\mathbb{R}^{d}\) is right-continuous with left limits;
			\item[(ii)] The process \(t \mapsto \eta_{t}\) is \(\mathcal{F}_{t}\)-adapted and satisfies \(\sup_{t\in[0,T]}\mathbb{E}[|\eta_{t}|^{p}]<\infty\).
		\end{itemize}
		Based on the distance \(\rho_{\infty}\) on \(\mathcal{D}([0,T];\mathbb{R}^{d})\), for each \(p\geq 1,\) with $\lambda$ as above, we define a metric on \(\mathscr{Y}_{T}(p)\) by  
		\[
		\boldsymbol{\rho}_{p,\infty}(\eta,\tilde{\eta}) := \bigg[\mathbb{E}\rho^{p}_{\infty}(\eta,\tilde{\eta})\bigg]^{1/p}, \quad \eta, \tilde{\eta} \in \mathscr{Y}_{T}(p). 
		\]
		The completeness of $(\mathscr{Y}_T(p), \boldsymbol{\rho}_{p,\infty})$ follows essentially from Lemma~2.3 in Shao et al.~\cite{Shao2024}.		
		\begin{enumerate}
			\item[${\bf (2)}$] \textbf{Freezing the distribution}
		\end{enumerate}
		For a given \(\eta \in \mathscr{Y}_T(2)\), denote \(\mu_t = \mathcal{L}(\eta_t | \mathcal{F}_t^0)\), we construct the following classic SDE
		\begin{equation}\label{rem 2.1}
			Y_{t}^{\mu}=Y_{0}^{\mu}+\int_{0}^{t}b(Y_{s}^{\mu},\mu_{s}){{\rm{d}}}s+\int_{0}^{t}\sigma(Y_{s}^{\mu},\mu_{s}){{\rm{d}}}W_{s}+\int_{0}^{t}c(Y_{s}^{\mu},\mu_{s}){{\rm{d}}}Z_{s}+\int_{0}^{t}\sigma^{0}(Y_{s}^{\mu},\mu_{s}){{\rm{d}}}W_{s}^{0},
		\end{equation}
		almost surely for any \( t \in [0, T],\) with \(Y_0^\mu = X_0\). According to \cite[Theorem 2.1\hspace{-0.1em}]{Neelima2020}, the classical SDE \eqref{rem 2.1} admits a unique solution with \( \sup_{t \in [0,T]} \mathbb{E}\bigl[\lvert Y_t^\mu \rvert^2\bigr] < \infty. \) Moreover, we have \((Y_t^\mu)_{t \in [0,T]} \in \mathscr{Y}_T(2)\).
		\begin{enumerate}
			\item[${\bf (3)}$] \textbf{Constructing a contractive mapping}
		\end{enumerate}
		Define the mapping  
		\[
		\Phi : \mathscr{Y}_T(2) \longrightarrow \mathscr{Y}_T(2), \quad \eta \mapsto Y^\mu.
		\]
		We show that \(\Phi\) is a strict contraction in \((\mathscr{Y}_T(2), \boldsymbol{\rho}_{2,\infty})\). For another \(\tilde{\eta} \in \mathscr{Y}_T(2)\), let \((Y_t^\nu)_{t \in [0,T]}\) solve SDE \eqref{rem 2.1} by replacing \(\mu_t = \mathcal{L}(\eta_t | \mathcal{F}_t^0)\) there with \(\nu_t = \mathcal{L}(\tilde{\eta}_t | \mathcal{F}_t^0)\) for \(t\in [0,T]\). 
		Applying It\^o's formula to \(|Y_{t}^{\mu}-Y_{t}^{\nu}|^{2} \),  taking the supremum over \([0,t] \subset [0,T]\) and expectation, and using the Assumption \textbf{(A2)}, we get
		\begin{equation}\label{2}
			\begin{aligned}
				&\mathbb{E}\Big[\sup_{0\le s\le t}\big|Y_s^\mu - Y_s^\nu\big|^2\Big]\\
				& \leq \mathbb{E}\sup_{0\le r\le t}\int_0^r\Big[ 2 \langle Y_s^\mu - Y_s^{\nu}, b(Y_s^\mu, \mu_s) - b(Y_s^{\nu}, \nu_s)\rangle \\
				&\quad+ |\sigma(Y_s^\mu, \mu_s) - \sigma(Y_s^{\nu}, \nu_s)|^2 + |\sigma^{0}(Y_s^\mu, \mu_s) - \sigma^{0}(Y_s^{\nu}, \nu_s)|^2 \Big]{\rm{d}}s \\
				&\quad+ \mathbb{E}[\sup_{0\le r\le t}\int_0^r \int_{\mathbb{R}_0^d} |(c(Y_s^\mu, \mu_s) - c(Y_s^{\nu}, \nu_s)) z|^2 \nu({\rm{d}}z) {\rm{d}}s] \\
				&\quad + 2\, \mathbb{E}\left[\sup_{0 \le r \le t} \int_0^r \left\langle Y_s^\mu - Y_s^\nu, (\sigma(Y_s^\mu, \mu_s) - \sigma(Y_s^\nu, \nu_s)) \mathrm{d}W_s \right\rangle \right] \\
				&\quad + 2\, \mathbb{E}\left[\sup_{0 \le r \le t} \int_0^r \left\langle Y_s^\mu - Y_s^\nu, (\sigma^0(Y_s^\mu, \mu_s) - \sigma^0(Y_s^\nu, \nu_s)) \mathrm{d}W_s^0 \right\rangle \right] \\
				&\quad + 2\, \mathbb{E}\left[\sup_{0 \le r \le t} \int_0^r \int_{\mathbb{R}_0^d} \Big( |Y_{s^-}^\mu - Y_{s^-}^\nu + \big( c(Y_{s^-}^\mu, \mu_s) - c(Y_{s^-}^\nu, \nu_s) \big)\, z|^2 - |Y_{s^-}^\mu - Y_{s^-}^\nu|^2 \Big) \widetilde{N}(\mathrm{d}s, \mathrm{d}z) \right]\\
				&\le \mathbb{E}\Big[\!\sup_{0 \le r \le t}\int_0^r \big(L_1|Y_s^\mu - Y_s^\nu|^2 + L_2\mathcal{W}_{2}^{2}(\mu_s,\nu_s)\big)\,\mathrm{d}s\Big]\\
				&\quad + 2\, \mathbb{E}\left[\sup_{0 \le r \le t} \int_0^r \left\langle Y_s^\mu - Y_s^\nu, (\sigma(Y_s^\mu, \mu_s) - \sigma(Y_s^\nu, \nu_s)) \mathrm{d}W_s \right\rangle \right] \\
				&\quad + 2\, \mathbb{E}\left[\sup_{0 \le r \le t} \int_0^r \left\langle Y_s^\mu - Y_s^\nu, (\sigma^0(Y_s^\mu, \mu_s) - \sigma^0(Y_s^\nu, \nu_s)) \mathrm{d}W_s^0 \right\rangle \right] \\
				&\quad + 2\, \mathbb{E}\left[\sup_{0 \le r \le t} \int_0^r \int_{\mathbb{R}_0^d} \Big( |Y_{s^-}^\mu - Y_{s^-}^\nu + \big( c(Y_{s^-}^\mu, \mu_s) - c(Y_{s^-}^\nu, \nu_s) \big)\, z|^2 - |Y_{s^-}^\mu - Y_{s^-}^\nu|^2 \Big) \widetilde{N}(\mathrm{d}s, \mathrm{d}z) \right].
			\end{aligned}
		\end{equation}
		Moreover, due to
		\[
		\mathbb{E}[|Y_s^\mu-Y_s^\nu|^2]\le \mathbb{E}\Big[\sup_{0\le r\le s}|Y_r^\mu-Y_r^\nu|^2\Big],\qquad
		\mathbb{E}[\mathcal W_2^2(\mu_s,\nu_s)]\le \mathbb{E}\Big[\sup_{0\le r\le s}\mathcal W_2^2(\mu_r,\nu_r)\Big],
		\]
		we obtain
		\begin{align*}
			&\mathbb{E}\Bigg[\int_0^t \Big(\tilde L_1 |Y_s^\mu-Y_s^\nu|^2+\tilde L_2 \mathcal W_2^2(\mu_s,\nu_s)\Big)\,\mathrm{d}s\Bigg]\\
			&= \int_0^t \Big(\tilde L_1\,\mathbb{E}[|Y_s^\mu-Y_s^\nu|^2]+\tilde L_2\,\mathbb{E}[\mathcal W_2^2(\mu_s,\nu_s)]\Big)\,\mathrm{d}s\\
			&\le \int_0^t \Big(\tilde L_1\,\mathbb{E}[\sup_{0\le r\le s}|Y_r^\mu-Y_r^\nu|^2]
			+\tilde L_2\,\mathbb{E}[\sup_{0\le r\le s}\mathcal W_2^2(\mu_r,\nu_r)]\Big)\,\mathrm{d}s.
		\end{align*}
		Therefore,
		\begin{align*}
			& C_1\,\mathbb{E}\Bigg[\Bigg(\int_0^t |Y_s^\mu - Y_s^\nu|^2
			\Big(\tilde L_1 |Y_s^\mu - Y_s^\nu|^2+\tilde L_2 \mathcal W_2^2(\mu_s,\nu_s)\Big)\,\mathrm{d}s\Bigg)^{1/2}\Bigg]\\
			&\quad \le C_1\Big(\mathbb{E}[\sup_{0\le r\le t}|Y_r^\mu-Y_r^\nu|^2]\Big)^{1/2}
			\Bigg(\int_0^t \Big(\tilde L_1\,\mathbb{E}[\sup_{0\le r\le s}|Y_r^\mu-Y_r^\nu|^2]
			+\tilde L_2\,\mathbb{E}[\sup_{0\le r\le s}\mathcal W_2^2(\mu_r,\nu_r)]\Big)\,\mathrm{d}s\Bigg)^{1/2}.
		\end{align*}
		Applying Young's inequality $uv\le \tfrac{1}{2}\epsilon u^2+\tfrac{1}{2\epsilon}v^2$, let \(\epsilon =\frac{1}{4C_1}\) we have
		\begin{align*}
			& C_1\,\mathbb{E}\Bigg[\Bigg(\int_0^t |Y_s^\mu - Y_s^\nu|^2
			\Big(\tilde L_1 |Y_s^\mu - Y_s^\nu|^2+\tilde L_2 \mathcal W_2^2(\mu_s,\nu_s)\Big)\,\mathrm{d}s\Bigg)^{1/2}\Bigg]\\
			&\quad \le \frac{C_1\epsilon}{2}\,\mathbb{E}\Big[\sup_{0\le r\le t}|Y_r^\mu-Y_r^\nu|^2\Big]
			+\frac{C_1}{2\epsilon}\int_0^t \Big(\tilde L_1\,\mathbb{E}[\sup_{0\le r\le s}|Y_r^\mu-Y_r^\nu|^2]
			+\tilde L_2\,\mathbb{E}[\sup_{0\le r\le s}\mathcal W_2^2(\mu_r,\nu_r)]\Big)\,\mathrm{d}s\\
			&\quad \le \frac{1}{8}\,\mathbb{E}\Big[\sup_{0\le r\le t}|Y_r^\mu-Y_r^\nu|^2\Big]
			+2{C_1}^2\int_0^t \Big(\tilde L_1\,\mathbb{E}[\sup_{0\le r\le s}|Y_r^\mu-Y_r^\nu|^2]
			+\tilde L_2\,\mathbb{E}[\sup_{0\le r\le s}\mathcal W_2^2(\mu_r,\nu_r)]\Big)\,\mathrm{d}s.
		\end{align*}
		Finally, using
		\begin{align*}
			&\tilde L_1\,\mathbb{E}[\sup_{0\le r\le s}|Y_r^\mu-Y_r^\nu|^2]
			+\tilde L_2\,\mathbb{E}[\sup_{0\le r\le s}\mathcal W_2^2(\mu_r,\nu_r)]\\
			&\quad \le \max\{\tilde L_1,\tilde L_2\}\Big(\mathbb{E}[\sup_{0\le r\le s}|Y_r^\mu-Y_r^\nu|^2]
			+\mathbb{E}[\sup_{0\le r\le s}\mathcal W_2^2(\mu_r,\nu_r)]\Big),
		\end{align*}
		we obtain
		\begin{equation}\label{a}
			\begin{aligned}
				& C_1\,\mathbb{E}\Bigg[\Bigg(\int_0^t |Y_s^\mu - Y_s^\nu|^2
				\Big(\tilde L_1 |Y_s^\mu - Y_s^\nu|^2+\tilde L_2 \mathcal W_2^2(\mu_s,\nu_s)\Big)\,\mathrm{d}s\Bigg)^{1/2}\Bigg]\\
				& \le \frac{1}{8}\,\mathbb{E}\Big[\sup_{0\le r\le t}|Y_r^\mu-Y_r^\nu|^2\Big]
				+ 2{C_1}^2\,\max\{\tilde L_1,\tilde L_2\}\\
				&\quad \times \int_0^t\Big(\mathbb{E}[\sup_{0\le r\le s}|Y_r^\mu-Y_r^\nu|^2]
				+\mathbb{E}[\sup_{0\le r\le s}\mathcal W_2^2(\mu_r,\nu_r)]\Big)\,\mathrm{d}s.
			\end{aligned}
		\end{equation}
		By using BDG's inequality and \eqref{a}, we have
		\begin{equation}\label{BDG1}
			\begin{aligned}
				&\mathbb{E} \left[ \sup_{0\le r \le t} \left| \int_0^r \langle Y_s^\mu - Y_s^\nu, (\sigma(Y_s^\mu, \mu_s) - \sigma(Y_s^\nu, \nu_s)) \, \mathrm{d}W_s \rangle \right| \right] \\
				&\quad \le C_1\, \mathbb{E} \left[ \left( \int_0^t |Y_s^\mu - Y_s^\nu|^2 \cdot |\sigma(Y_s^\mu, \mu_s) - \sigma(Y_s^\nu, \nu_s)|^2 \,\mathrm{d}s \right)^{1/2} \right] \\
				&\quad \le C_1\, \mathbb{E} \left[ \left( \int_0^t |Y_s^\mu - Y_s^\nu|^2 \cdot \left(\tilde{L}_1|Y_s^\mu - Y_s^\nu|^2+\tilde{L}_2\mathcal{W}_{2}^{2}(\mu_s,\nu_s)\right) \,\mathrm{d}s \right)^{1/2} \right]\\
				&\quad \leq \frac{1}{8}\,\mathbb{E}\Big[\sup_{0\le r\le t}|Y_r^\mu-Y_r^\nu|^2\Big]+ K_1\int_{0}^{t}\left(\mathbb{E}\Big[\sup_{0\le r\le s}|Y_r^\mu - Y_r^\nu|^2\Big]+\mathbb{E}\Big[\sup_{0\le r\le s}\mathcal{W}_{2}^{2}(\mu_r,\nu_r)\Big]\right)\,\mathrm{d}s,
			\end{aligned}
		\end{equation}
		where \(K_1 \;=\; 2{C_1}^2\cdot\,\max\{\tilde L_1,\tilde L_2\}\). In the same way for the term driven by \(W\), we have
		\begin{equation}\label{BDG2}
			\begin{aligned}
				&\mathbb{E}\left[\sup_{0 \le r \le t} \int_0^r \left\langle Y_s^\mu - Y_s^\nu, (\sigma^0(Y_s^\mu, \mu_s) - \sigma^0(Y_s^\nu, \nu_s)) \mathrm{d}W_s^0 \right\rangle \right] \\
				&\quad \le C_2\, \mathbb{E} \left[ \left( \int_0^t |Y_s^\mu - Y_s^\nu|^2 \cdot |\sigma^0(Y_s^\mu, \mu_s) - \sigma^0(Y_s^\nu, \nu_s)|^2 \,\mathrm{d}s \right)^{1/2} \right] \\
				&\quad \le C_2\, \mathbb{E} \left[ \left( \int_0^t |Y_s^\mu - Y_s^\nu|^2 \cdot \left(\tilde{L}_1 |Y_s^\mu - Y_s^\nu|^2 + \tilde{L}_2 \mathcal{W}_2^2(\mu_s, \nu_s)\right) \,\mathrm{d}s \right)^{1/2} \right] \\
				&\quad \le \frac{1}{8}\,\mathbb{E}\Big[ \sup_{0\le r\le t} |Y_r^\mu - Y_r^\nu|^2 \Big] 
				+ K_2 \int_0^t \Big( \mathbb{E}\big[\sup_{0\le r\le s}|Y_r^\mu - Y_r^\nu|^2\big] 
				+ \mathbb{E}\big[\sup_{0\le r\le s}\mathcal{W}_2^2(\mu_r,\nu_r)\big] \Big)\,\mathrm{d}s,
			\end{aligned}
		\end{equation}
		where \(K_2 = 2C_2^2 \cdot \max\{\tilde{L}_1,\tilde{L}_2\}\).\\
		As for the jump term, let
		\[
		M_r:=\int_0^r \int_{\mathbb{R}_0^d} \Big( |Y_{s^-}^\mu - Y_{s^-}^\nu + \big( c(Y_{s^-}^\mu, \mu_s) - c(Y_{s^-}^\nu, \nu_s) \big)\, z|^2 - |Y_{s^-}^\mu - Y_{s^-}^\nu|^2 \Big) \widetilde{N}(\mathrm{d}s, \mathrm{d}z),\qquad r\in[0,t].
		\]
		Note that
		\[
		\begin{aligned}
			&\Big|\big|Y_{s^-}^\mu-Y_{s^-}^\nu +\big(c(Y_{s^-}^\mu,\mu_s)-c(Y_{s^-}^\nu,\nu_s)\big)\,z\big|^2 - |Y_{s^-}^\mu-Y_{s^-}^\nu|^2\Big| \\
			&\quad \le 2\,|Y_{s^-}^\mu-Y_{s^-}^\nu|\,\cdot\big|c(Y_{s^-}^\mu,\mu_s)-c(Y_{s^-}^\nu,\nu_s)\big|\,|z|  + \big|c(Y_{s^-}^\mu,\mu_s)-c(Y_{s^-}^\nu,\nu_s)\big|^2\,|z|^2.
		\end{aligned}
		\]
		By the Kunita's inequality (see~\cite[Theorem 4.4.23\hspace{-0.01em}]{Applebaum2009}) for compensated Poisson integrals there exists a constant $C_3>0$ such that
		\[
		\begin{aligned}
			&\mathbb{E}\Big[\sup_{0\le r\le t}|M_r|\Big]\le C_3\Bigg\{\mathbb{E}\Big[\Big(\int_0^t\int_{\mathbb R^d_0}
			\Big|2\big\langle Y_{s^-}^\mu-Y_{s^-}^\nu,\; c(Y_{s^-}^\mu,\mu_s)-c(Y_{s^-}^\nu,\nu_s)\big\rangle\,z  \\
			& + \big|c(Y_{s^-}^\mu,\mu_s)-c(Y_{s^-}^\nu,\nu_s)\big|^2\,|z|^2\Big|^2
			\,\nu(\mathrm{d}z)\,\mathrm{d}s\Big)^{1/2}\Big] +\;\mathbb{E}\Big[\int_0^t\int_{\mathbb R^d_0}
			\Big|2\big\langle Y_{s^-}^\mu-Y_{s^-}^\nu,\; c(Y_{s^-}^\mu,\mu_s)-c(Y_{s^-}^\nu,\nu_s)\big\rangle\,z   \\
			& + \big|c(Y_{s^-}^\mu,\mu_s)-c(Y_{s^-}^\nu,\nu_s)\big|^2\,|z|^2\Big|
			\,\nu(\mathrm{d}z)\,\mathrm{d}s\Big]\Bigg\}.
		\end{aligned}
		\]
		Using \(|a+b|^2 \le 2|a|^2 + 2|b|^2\), \(|a+b| \le |a|+|b|\) and the Assumption \textbf{(A7)}, we get	
		\[
		\begin{aligned}
			&\int_0^t \int_{\mathbb R^d_0} \Big|2 \langle Y_{s^-}^\mu - Y_{s^-}^\nu, c(Y_{s^-}^\mu,\mu_s)-c(Y_{s^-}^\nu,\nu_s) \rangle z + |c(Y_{s^-}^\mu,\mu_s)-c(Y_{s^-}^\nu,\nu_s|^2 |z|^2 \Big|^2 \nu({\rm{d}}z) {\rm{d}}s \\
			&\quad \le \int_0^t \int_{\mathbb R^d_0} 8 |Y_{s^-}^\mu - Y_{s^-}^\nu|^2\cdot |c(Y_{s^-}^\mu,\mu_s)-c(Y_{s^-}^\nu,\nu_s)|^2 |z|^2 + 2 |c(Y_{s^-}^\mu,\mu_s)-c(Y_{s^-}^\nu,\nu_s)|^4 |z|^4 \, \nu({\rm{d}}z) {\rm{d}}s\\
			&\quad \le \int_0^t 8 |Y_{s^-}^\mu - Y_{s^-}^\nu|^2 \big( \tilde{L}_1 |Y_{s^-}^\mu - Y_{s^-}^\nu|^2 + \tilde{L}_2 \mathcal W_2^2(\mu_s,\nu_s) \big) {\rm{d}}s \\
			&\qquad + 2 \big( \tilde{L}_1 |Y_{s^-}^\mu - Y_{s^-}^\nu|^2 + \tilde{L}_2 \mathcal W_2^2(\mu_s,\nu_s) \big)^2 \bigg(\int_{\mathbb R^d_0}|z|^2\nu(\mathrm{d}z)\bigg)^{-2}\int_{\mathbb R^d_0} |z|^4 \nu({\rm{d}}z) {\rm{d}}s, 
		\end{aligned}
		\]	
		and	
		\[
		\begin{aligned}
			&\int_0^t \int_{\mathbb R^d_0} \Big|2 \langle Y_{s^-}^\mu - Y_{s^-}^\nu, c(Y_{s^-}^\mu,\mu_s)-c(Y_{s^-}^\nu,\nu_s) \rangle z + |c(Y_{s^-}^\mu,\mu_s)-c(Y_{s^-}^\nu,\nu_s|^2 |z|^2 \Big| \nu({\rm{d}}z) {\rm{d}}s \\
			&\quad \le \int_0^t \int_{\mathbb R^d_0} 2 |Y_{s^-}^\mu - Y_{s^-}^\nu|\cdot |c(Y_{s^-}^\mu,\mu_s)-c(Y_{s^-}^\nu,\nu_s)| |z| + |c(Y_{s^-}^\mu,\mu_s)-c(Y_{s^-}^\nu,\nu_s|^2 |z|^2 \, \nu({\rm{d}}z) {\rm{d}}s\\
			&\quad \le \int_0^t 2 |Y_{s^-}^\mu - Y_{s^-}^\nu| \sqrt{\tilde{L}_1 |Y_{s^-}^\mu - Y_{s^-}^\nu|^2 + \tilde{L}_2 \mathcal W_2^2(\mu_s,\nu_s)} \bigg(\int_{\mathbb R^d_0}|z|^2\nu(\mathrm{d}z)\bigg)^{-1/2}\int_{\mathbb R^d_0} |z| \nu({\rm{d}}z) {\rm{d}}s \\
			&\qquad + \int_0^t (\tilde{L}_1 |Y_{s^-}^\mu - Y_{s^-}^\nu|^2 + \tilde{L}_2 \mathcal W_2^2(\mu_s,\nu_s)){\rm{d}}s.
		\end{aligned}
		\]
		Hence, similar to the term driven by \(W\),
		\begin{equation}\label{BDG3}
			\begin{aligned}
				&\mathbb{E}\Big[\sup_{0\le r\le t}\Big|\int_0^r\int_{\mathbb R^d_0}
				\langle Y_s^\mu - Y_s^\nu,\; c(Y_s^\mu,\mu_s,z)-c(Y_s^\nu,\nu_s,z)\rangle
				\,\tilde N(\mathrm{d}s,\mathrm{d}z)\Big|\Big] \\
				&\quad \le \frac{1}{8}\,\mathbb{E}\Big[\sup_{0\le r\le t}|Y_r^\mu-Y_r^\nu|^2\Big]  + K_3\int_0^t\Big(\mathbb{E}[\sup_{0\le r\le s}|Y_r^\mu-Y_r^\nu|^2]
				+\mathbb{E}[\sup_{0\le r\le s}\mathcal W_2^2(\mu_r,\nu_r)]\Big)\,\mathrm{d}s,
			\end{aligned}
		\end{equation}
		where 
		\(K_3\) depends on \(\tilde{L}_1,\tilde{L}_2\) and the moments \(\int_{\mathbb R^d_0} |z| \nu({\rm{d}}z), \int_{\mathbb R^d_0} |z|^2 \nu({\rm{d}}z), \int_{\mathbb R^d_0} |z|^4 \nu({\rm{d}}z)\).			
		Substituting \eqref{BDG1}, \eqref{BDG2}, \eqref{BDG3} into \eqref{2}, we obtain
		\begin{equation}\label{sup}
			\begin{aligned}
				&\mathbb{E}\Big[\sup_{0\le s\le t}\big|Y_s^\mu - Y_s^\nu\big|^2\Big]\\
				&\le \mathbb{E}\Big[\!\sup_{0 \le r \le t}\int_0^r \big(L_1|Y_s^\mu - Y_s^\nu|^2 + L_2\mathcal{W}_{2}^{2}(\mu_s,\nu_s)\big)\,\mathrm{d}s\Big]\\
				&\quad + 2\, \mathbb{E}\left[\sup_{0 \le r \le t} \int_0^r \left\langle Y_s^\mu - Y_s^\nu, (\sigma(Y_s^\mu, \mu_s) - \sigma(Y_s^\nu, \nu_s)) \mathrm{d}W_s \right\rangle \right] \\
				&\quad + 2\, \mathbb{E}\left[\sup_{0 \le r \le t} \int_0^r \left\langle Y_s^\mu - Y_s^\nu, (\sigma^0(Y_s^\mu, \mu_s) - \sigma^0(Y_s^\nu, \nu_s)) \mathrm{d}W_s^0 \right\rangle \right] \\
				&\quad + 2\, \mathbb{E}\left[\sup_{0 \le r \le t} \int_0^r \int_{\mathbb{R}_0^d} \Big( |Y_{s^-}^\mu - Y_{s^-}^\nu + \big( c(Y_{s^-}^\mu, \mu_s) - c(Y_{s^-}^\nu, \nu_s) \big)\, z|^2 - |Y_{s^-}^\mu - Y_{s^-}^\nu|^2 \Big) \widetilde{N}(\mathrm{d}s, \mathrm{d}z) \right]\\
				&\le \mathbb{E}\Big[\!\sup_{0 \le r \le t}\int_0^r \big(L_1|Y_s^\mu - Y_s^\nu|^2 + L_2\mathcal{W}_{2}^{2}(\mu_s,\nu_s)\big)\,\mathrm{d}s\Big]+ \frac{3}{4}\,\mathbb{E}\Big[\sup_{0\le r\le t}|Y_r^\mu-Y_r^\nu|^2\Big]\\
				&\quad + 2(K_1+K_2+K_3)\int_{0}^{t}\left(\mathbb{E}\Big[\sup_{0\le r\le s}|Y_r^\mu - Y_r^\nu|^2\Big]+\mathbb{E}\Big[\sup_{0\le r\le s}\mathcal{W}_{2}^{2}(\mu_r,\nu_r)\Big]\right)\,\mathrm{d}s\\
				&\leq \frac{3}{4}\,\mathbb{E}\Big[\sup_{0\le r\le t}|Y_r^\mu-Y_r^\nu|^2\Big] + \frac{\hat{C}}{4}\int_{0}^{t}\left(\mathbb{E}[\sup_{0\le r\le s}|Y_r^\mu - Y_r^\nu|^2+\sup_{0\le r\le s}\mathcal{W}_{2}^{2}(\mu_r,\nu_r)]\right){\rm{d}}s,
			\end{aligned}
		\end{equation}
		where $\hat{C}:=\hat{C}(L_1,L_2,K_1,K_2,K_3)$.
		Hence,
		\begin{align*}
			\mathbb{E}\Big[\sup_{0\le s\le t}\big|Y_s^\mu - Y_s^\nu\big|^2\Big]&\leq \hat{C}\int_{0}^{t}\left(\mathbb{E}[\sup_{0\le r\le s}|Y_r^\mu - Y_r^\nu|^2+\sup_{0\le r\le s}\mathcal{W}_{2}^{2}(\mu_r,\nu_r)]\right){\rm{d}}s\\
			&\leq \hat{C}t\cdot \mathbb{E}[\sup_{0\le r\le t}|Y_r^\mu - Y_r^\nu|^2]+\hat{C}t\cdot \mathbb{E}[\sup_{0\le r\le t}\mathcal{W}_{2}^{2}(\mu_r,\nu_r)].
		\end{align*}
		Therefore, we have
		\begin{align*}
			\mathbb{E}\Big[\sup_{0\le s\le t}\big|Y_s^\mu - Y_s^\nu\big|^2\Big] &\leq \frac{\hat{C}t}{1-\hat{C}t} \cdot \mathbb{E}[\sup_{0\le r\le t}\mathcal{W}_{2}^{2}(\mu_r,\nu_r)] \\
			&\leq \frac{\hat{C}t}{1-\hat{C}t} \cdot \mathbb{E}[\sup_{0\le r\le t}|\eta_{r}-\tilde{\eta}_{r}|^2].
		\end{align*}
Combining the definition of the metric on the space 
$(\mathscr{Y}_T(2), \boldsymbol{\rho}_{2,\infty})$ and recalling that
$Y := \Phi(X)$ and $\tilde Y := \Phi(\tilde X)$, we obtain, for any $t\in[0,T]$,
\begin{align*}
	\boldsymbol{\rho}_{2,\infty}(\Phi(X),\Phi(\tilde{X}))
	&= \boldsymbol{\rho}_{2,\infty}(Y,\tilde{Y}) := \bigg[\mathbb{E}\sup_{0\le r\le t}|Y_{r}-\tilde{Y}_{r}|^{2}\bigg]^{1/2} \\
	&\leq \bigg(\frac{\hat{C}t}{1-\hat{C}t}\bigg)^{1/2}
	\bigg[\mathbb{E}\sup_{0\le r\le t}|X_{r}-\tilde{X}_{r}|^{2}\bigg]^{1/2} \\
	&=: \bigg(\frac{\hat{C}t}{1-\hat{C}t}\bigg)^{1/2}
	\boldsymbol{\rho}_{2,\infty}(X,\tilde{X}).
\end{align*}
		Choosing \(t=t_0\in(0,\frac{1}{2\hat{C}})\), then we have \(0<\frac{\hat{C}t_0}{1-\hat{C}t_0}<1\).
		
		By the contraction mapping principle on $[0,t_0]$, the mapping
		\[
		\Phi : \mathscr{Y}_T(2) \to \mathscr{Y}_T(2)
		\]
		is a strict contraction on the complete metric space \(\bigl(\mathscr{Y}_T(2), \boldsymbol{\rho}_{2,\infty}\bigr)\). Hence there exists a unique fixed point \(X^*\in\mathscr{Y}_T(2)\) such that
		\[
		\Phi(X^*) \;=\; X^*.
		\]
		By definition of \(\Phi\), \(X^*\) satisfies
		\begin{align*}
			X^*_t&= X^\mu_0
			+ \int_0^t b\bigl(X^*_s,\mathcal{L}(X^*_s\mid\F^0_s)\bigr)\,{{\rm{d}}}s + \int_0^t \sigma \bigl(X^*_s,\mathcal{L}(X^*_s\mid\F^0_s)\bigr)\,{{\rm{d}}}W_s
			+ \int_0^t c\bigl(X^*_s,\mathcal{L}(X^*_s\mid\F^0_s)\bigr)\,{{\rm{d}}}Z_s\\
			&+ \int_0^t \sigma^0\bigl(X^*_s,\mathcal{L}(X^*_s\mid\F^0_s)\bigr)\,{\rm{d}}W^0_s,
		\end{align*}
		which is precisely equation \eqref{eq:MKV-SDE}. Uniqueness follows immediately: if \(Y\in\mathscr{Y}_T(2)\) is any other solution of \eqref{eq:MKV-SDE}, then \(Y = \Phi(Y)\), and by uniqueness of the fixed point \(Y = X^*\). Therefore, SDE \eqref{eq:MKV-SDE} with initial distribution \(X^\mu_0\) has a unique solution. 
		
		We now iterate this argument to cover the whole interval $[0,T]$. Let
		$N:=\lceil T/t_0\rceil$ and for $k=1,\dots,N-1$ define the shifted time intervals
		$I_k:=[k t_0,(k+1)t_0\wedge T]$. Suppose that a unique solution
		$X^{(k-1)}$ has been constructed on $[0,k t_0]$. Use the terminal value
		$X^{(k-1)}_{k t_0}$ as the initial condition for the problem on the next interval $I_k$.
		Define the mapping $\Phi^{(k)}$ on the space $\mathscr{Y}_{I_k}(2)$ exactly as before but
		with initial time $k t_0$ and initial law induced by $X^{(k-1)}_{k t_0}$. The same local
		estimate (with the same constant $\hat C$) is valid on any interval of length $\le t_0$,
		hence $\Phi^{(k)}$ is again a contraction on $\mathscr{Y}_{I_k}(2)$ with contraction
		constant $q<1$. Therefore there exists a unique fixed point $X^{(k)}$ on $I_k$.
	    By constructing the solutions,  we have  $X^{(k)}_{k t_0}=X^{(k-1)}_{k t_0}$,
		so the pieces $\{X^{(k)}:k=0,\dots,N-1\}$ connect together to give a unique adapted solution
		$X$ on $[0,T]$ satisfying \eqref{eq:MKV-SDE}. Thus the existence on $[0,T]$ is proved.
		
		\begin{enumerate}
			\item[${\bf (4)}$] \textbf{Proving the boundedness of the second-order moment }
		\end{enumerate}
		Apply It\^o's formula to $ |X_{t}|^{2} $. Using Assumption \textbf{(A1)}, taking expectation and noting \[ \mathbb{E}[\mathcal{W}_{2}^{2}(\mathcal{L}^{1}(X_{s}),\delta_{0})]\leq\mathbb{E}|X_{s}|^{2}, \] we derive
		\[
		\mathbb{E}|X_{t}|^{2}\leq |x_{0}|^{2}+L\int_{0}^{t}(1+2\mathbb{E}|X_{s}|^{2}){\rm{d}}s.
		\]
		By Gr\"onwall's inequality, we obtain 
		\[
		\sup_{t\in[0,T]}\mathbb{E}|X_{t}|^{2}\leq \left(|x_{0}|^{2}+LT\right)e^{2LT}:=K.
		\]
		This establishes existence, uniqueness, and the boundedness of the second-order moment. 
	\end{proof} 
	We will need the following Lemma later.
		\begin{lem}\label{lemma}
		For any $s \geq 0$, we have
		\begin{equation}
			\mathcal{W}^{2}_2(\mathcal{L}^{1}_{X_s}, \delta_{0}) = \mathbb{E}[|X_{s}|^{2} \mid \mathcal{F}^0],
		\end{equation}
		where $\mathcal{L}^{1}_{X_s}$ denotes the conditional law of $X_s$ given $\mathcal{F}^0$.
	\end{lem}
	\begin{proof}
		Recall that for two probability measures $\mu$ and $\nu$ on $\mathbb{R}^d$, the 2-Wasserstein distance is defined as
		\[
		\mathcal{W}_2^2(\mu, \nu) = \inf_{\pi \in \Pi(\mu, \nu)} \iint_{\mathbb{R}^d \times \mathbb{R}^d} |x - y|^2  {\rm{d}}\pi(x, y),
		\]
		where $\Pi(\mu, \nu)$ denotes the set of all couplings of $\mu$ and $\nu$.
		
		In our case, $\nu = \delta_0$ is the Dirac measure at the origin. Since $\delta_0$ is a deterministic measure, there is only one possible coupling $\pi$ between $\mathcal{L}^{1}_{X_s}$ and $\delta_0$, namely the product measure
		\[
		\pi = \mathcal{L}^{1}_{X_s} \otimes \delta_0.
		\]
		Now compute the cost with this specific coupling
		\[
		\begin{aligned}
			\mathcal{W}_2^2(\mathcal{L}^{1}_{X_s}, \delta_0) 
			&= \iint_{\mathbb{R}^d \times \mathbb{R}^d} |x - y|^2  {\rm{d}}\pi(x, y) \\
			&= \iint_{\mathbb{R}^d \times \mathbb{R}^d} |x - y|^2  {\rm{d}}(\mathcal{L}^{1}_{X_s} \otimes \delta_0)(x, y) \\
			&= \int_{\mathbb{R}^d} \left( \int_{\mathbb{R}^d} |x - y|^2  \delta_0({\rm{d}}y) \right) \mathcal{L}^{1}_{X_s}({\rm{d}}x).
		\end{aligned}
		\]
		Since $\delta_0$ is concentrated at $y = 0$, the inner integral evaluates to
		\[
		\int_{\mathbb{R}^d} |x - y|^2  \delta_0({\rm{d}}y) = |x - 0|^2 = |x|^2.
		\]
		Therefore,
		\[
		\mathcal{W}_2^2(\mathcal{L}^{1}_{X_s}, \delta_0) = \int_{\mathbb{R}^d} |x|^2  \mathcal{L}^{1}_{X_s}({\rm{d}}x) = \mathbb{E}[|X_s|^2 \mid \mathcal{F}^0].
		\]
		This completes the proof.
	\end{proof}
	\begin{prp}\label{prp:2.3}
		Let assumptions \textup{\textbf{(A1)-(A7)}} hold. Based on Theorem \ref{00}, we can further obtain for each $p \in [2, p_0]$, one can find a constant $C_p > 0$ such that for all $t \ge 0$, 
		\begin{equation}\label{2.1}
			\mathbb{E} [|X_t|^p] \leq
			\begin{cases}
				C_p(1 + e^{\gamma p t}) & \text{if } \gamma \neq 0, \\
				C_p(1 + t)^{p/2} & \text{if } \gamma = 0, p = 2 \text{ or } \gamma = 0, \gamma_2 > 0, p \in (2, p_0], \\
				C_p(1 + t)^p & \text{if } \gamma = 0, \gamma_2 = 0, p \in (2, p_0],
			\end{cases}
		\end{equation}
		where $\gamma = \gamma_1 + \gamma_2$.\\
		\(\text{If }\gamma<0,\text{ then }
		\sup_{t\ge0}\mathbb{E}\bigl[|X_t|^p\bigr]\le2C_p.\)		
	\end{prp}
	
	\begin{proof}
		We hereby follow the methodology introduced in \cite[Proposition 2.3 \hspace{-0.1em}]{Tran2025}.\\
		\textbf{Step 1:}
		It follows from Theorem \ref{00} that $\sup_{t \in [0,T]} \mathbb{E} [|X_t|^2] \leq K,K := K(\mathbb{E}|X^\mu_0|^2, d, L, L_1, \gamma_1, \gamma_2, T).$\\
		\textbf{Step 2:}
		We first demonstrate that for any even natural number $p \in [2, p_0]$ and for any $T > 0$, the following holds,
		\begin{equation}\label{2.2}
			\sup_{t \in [0,T]} \mathbb{E}[|X_t|^p] \leq C(T,p).
		\end{equation}
		Observe that \eqref{2.2} is valid for the case $p=2$ as established in Step~1.  
		Assume now, as an induction hypothesis, that \eqref{2.2} holds for every even integer $q \in [2,\, p-2]$, namely
		\begin{equation}\label{2.3}
			\sup_{t \in [0,T]} \mathbb{E}\big[\,|X_t|^q\,\big] \le C(T,q).
		\end{equation}
		Let $\lambda \in \mathbb{R}$ and consider any even integer $p$.  
		By applying It\^o's formula to the process $e^{-\lambda t}|X_t|^p$, we obtain
		\begin{equation}\label{2.4}
			\begin{aligned}
				&e^{-\lambda t} |X_t|^p = |x_0|^p + \int_0^t e^{-\lambda s} \Big[ 
				-\lambda |X_s|^p + p |X_s|^{p-2} \langle X_s, b(X_s, \mathcal{L}^{1}_{X_s}) \rangle + \frac{p}{2} |X_s|^{p-2} |\sigma(X_s, \mathcal{L}^{1}_{X_s})|^2 \\
				&+ \frac{p}{2} |X_s|^{p-2} |\sigma^{0}(X_s, \mathcal{L}^{1}_{X_s})|^2 + \frac{p(p-2)}{2} |X_s|^{p-4} \big| X_s^\top \sigma(X_s,\mathcal{L}^{1}_{X_s}) \big|^2  \\
				& + \frac{p(p-2)}{2} |X_s|^{p-4} \big| X_s^\top \sigma^{0}(X_s,\mathcal{L}^{1}_{X_s}) \big|^2 \Big] {{\rm{d}}}s  + p \int_0^t e^{-\lambda s}|X_s|^{p-2} \Big[ \langle X_s, \sigma(X_s, \mathcal{L}^{1}_{X_s}) {{\rm{d}}}W_s \rangle \\
				&+ \langle X_s, \sigma^{0}(X_s, \mathcal{L}^{1}_{X_s}) {{\rm{d}}}W^{0}_s \rangle \Big] 
				+ \int_0^t \int_{\mathbb{R}_0^d} e^{-\lambda s} \Big( |X_{s^-} + c(X_{s^-}, \mathcal{L}^{1}_{X_{s^-}}) z|^p - |X_{s^-}|^p \Big) \widetilde{N}({{\rm{d}}}s, {{\rm{d}}}z) \\
				& + \int_0^t \int_{\mathbb{R}_0^d} e^{-\lambda s} \Big( |X_{s} + c(X_{s}, \mathcal{L}^{1}_{X_{s}}) z|^p - |X_{s}|^p - p|X_{s}|^{p-2} \langle X_s, c(X_{s}, \mathcal{L}^{1}_{X_{s}}) z \rangle \Big) \nu({{\rm{d}}}z) {{\rm{d}}}s.
			\end{aligned}
		\end{equation}
		To handle the last integral in \eqref{2.4}, it suffices to apply the binomial theorem, which yields, for any $t \geq 0$,
		\begin{equation}\label{2.5}
			\begin{aligned}
				&|X_s + c(X_s, \mathcal{L}^{1}_{X_s}) z|^p 
				= \left( |X_s|^2 + |c(X_s, \mathcal{L}^{1}_{X_s}) z|^2 + 2 \langle X_s, c(X_s, \mathcal{L}^{1}_{X_s}) z \rangle \right)^{p/2}\\
				&= |X_s|^p + \frac{p}{2} |X_s|^{p-2} \left( |c(X_s, \mathcal{L}^{1}_{X_s}) z|^2 + 2 \langle X_s, c(X_s, \mathcal{L}^{1}_{X_s}) z \rangle \right)\\
				&\quad+ \sum_{i=2}^{p/2} \binom{p/2}{i} |X_s|^{p-2i} \left( |c(X_s, \mathcal{L}^{1}_{X_s}) z|^2 + 2 \langle X_s, c(X_s, \mathcal{L}^{1}_{X_s}) z \rangle \right)^i\\
				&= |X_s|^p + p |X_s|^{p-2} \langle X_s, c(X_s, \mathcal{L}^{1}_{X_s}) z \rangle + \frac{p}{2} |X_s|^{p-2} |c(X_s, \mathcal{L}^{1}_{X_s}) z|^2\\
				&\quad+ \sum_{i=2}^{p/2} \binom{p/2}{i} |X_s|^{p-2i} \left( |c(X_s, \mathcal{L}^{1}_{X_s}) z|^2 + 2 \langle X_s, c(X_s, \mathcal{L}^{1}_{X_s}) z \rangle \right)^i.
			\end{aligned}
		\end{equation}
		Next, by repeatedly applying the binomial theorem, using Assumption \textbf{(A5)}, the identity \eqref{lemma}
		and Young's inequality
		\[
		y|x|^{p-3} \leq \frac{1}{2}(|x|^{p-2} + y^{2}|x|^{p-4}),
		\]
		which holds for any $x \in \mathbb{R}^d$, $y > 0$, as well as the Binomial identities
		\[
		\sum_{j=0}^i \binom{i}{j} a^{j} = (1 + a)^{i}\quad \textup{and} \quad \sum_{j=0}^i \binom{i}{j} j a^{j} = i a (1 + a)^{i-1},
		\]
		which valid for any $a \in \mathbb{R}$, we obtain,
		\begin{align*}
			&|X_s|^{p-2i} \Big( |c(X_s, \mathcal{L}^1_{X_s}) z|^2 + 2 \langle X_s, c(X_s, \mathcal{L}^1_{X_s}) z \rangle \Big)^i \\
			&= |X_s|^{p-2i} \sum_{j=0}^i \binom{i}{j} |c(X_s, \mathcal{L}^1_{X_s}) z|^{2i-2j} 2^j \langle X_s, c(X_s, \mathcal{L}^1_{X_s}) z \rangle^j \\
			&\leq \sum_{j=0}^i \binom{i}{j} 2^j |X_s|^{p-2i+j} |c(X_s, \mathcal{L}^1_{X_s})|^{2i-j} |z|^{2i-2j} \\
			&= |c(X_s, \mathcal{L}^1_{X_s})|^2 \sum_{j=0}^i \binom{i}{j} 2^j |X_s|^{p-2i+j} |c(X_s, \mathcal{L}^1_{X_s})|^{2i-j-2} |z|^{2i-j} \\
			&\leq |c(X_s, \mathcal{L}^1_{X_s})|^2 \sum_{j=0}^i \binom{i}{j} 2^j |X_s|^{p-2i+j}|z|^{2i-j} L_3^{2i-j-2}(1+|X_s|+\mathcal{W}_2(\mathcal{L}^{1}_{X_s}, \delta_{0}))^{2i-j-2}\\
			&= |c(X_s, \mathcal{L}^1_{X_s})|^2 \sum_{j=0}^i \binom{i}{j} 2^j |X_s|^{p-2i+j} |z|^{2i-j} L_3^{2i-j-2}
			\Big( 1 + |X_s| + \sqrt{\mathbb{E}[|X_s|^2|\mathcal{F}^0]} \Big)^{2i-j-2} \\
			&= |c(X_s, \mathcal{L}^1_{X_s})|^2 \sum_{j=0}^i \binom{i}{j} 2^j |X_s|^{p-2i+j} |z|^{2i-j} L_3^{2i-j-2}
			\Bigg( |X_s|^{2i-j-2} \\
			&\quad+ (2i-j-2)
			\Big( 1 + \sqrt{\mathbb{E}[|X_s|^2|\mathcal{F}^0]} \Big) |X_s|^{2i-j-3}  + \sum_{k=2}^{2i-j-2} \binom{2i-j-2}{k}\\
			&\quad \times \Big( 1 + \sqrt{\mathbb{E}[|X_s|^2|\mathcal{F}^0]} \Big)^k |X_s|^{2i-j-2-k} \Bigg) \\
			&= |c(X_s,\mathcal{L}^1_{X_s})|^2 \sum_{j=0}^i \binom{i}{j} 2^j |z|^{2i-j} L_3^{2i-j-2}
			\Bigg( |X_s|^{p-2} + (2i-j-2)
			\Big( 1 + \sqrt{\mathbb{E}[|X_s|^2|\mathcal{F}^0]} \Big) |X_s|^{p-3} \\
			&\quad+ \sum_{k=2}^{2i-j-2} \binom{2i-j-2}{k}
			\Big( 1 + \sqrt{\mathbb{E}[|X_s|^2|\mathcal{F}^0]} \Big)^k |X_s|^{p-2-k} \Bigg) \\
			&\leq |c(X_s, \mathcal{L}^1_{X_s})|^2 \sum_{j=0}^i \binom{i}{j} 2^j |z|^{2i-j} L_3^{2i-j-2}
			\Bigg( |X_s|^{p-2} + \frac{(2i-j-2)}{2}
			\Big( |X_s|^{p-2} \\
			&\quad+ \Big( 1 + \sqrt{\mathbb{E}[|X_s|^2|\mathcal{F}^0]} \Big)^2 |X_s|^{p-4} \Bigg) + \sum_{k=2}^{2i-j-2} \binom{2i-j-2}{k}
			\Big( 1 + \sqrt{\mathbb{E}[|X_s|^2|\mathcal{F}^0]} \Big)^k |X_s|^{p-2-k} \Bigg) \\
			&= |c(X_s, \mathcal{L}^1_{X_s})|^2 |X_s|^{p-2} \sum_{j=0}^i \binom{i}{j} 2^j |z|^{2i-j} L_3^{2i-j-2}(i-\frac{j}{2})\\
			&\quad+|c(X_s, \mathcal{L}^1_{X_s})|^2\sum_{j=0}^i \binom{i}{j} 2^j |z|^{2i-j} L_3^{2i-j-2}\Bigg(\Big( i-\frac{j}{2} - 1 \Big)\Big( 1 + \sqrt{\mathbb{E}[|X_s|^2|\mathcal{F}^0]} \Big)^2 |X_s|^{p-4} \\
			&\quad+ \sum_{k=2}^{2i-j-2} \binom{2i-j-2}{k}
			\Big( 1 + \sqrt{\mathbb{E}[|X_s|^2|\mathcal{F}^0]} \Big)^k |X_s|^{p-2-k} \Bigg)\\
			&= |c(X_s, \mathcal{L}^1_{X_s})|^2 |X_s|^{p-2}\big(\frac{i}{L_3^{2}}(L_3^{2}|z|^{2}+2L_3|z|)^{i}-\frac{|z|}{L_3}i(L_3^{2}|z|^{2}+2L_3|z|)^{i}\big)\\
			&\quad+|c(X_s, \mathcal{L}^1_{X_s})|^2\sum_{j=0}^i \binom{i}{j} 2^j |z|^{2i-j} L_3^{2i-j-2}\Bigg(\Big( i-\frac{j}{2} - 1 \Big)\Big( 1 + \sqrt{\mathbb{E}[|X_s|^2|\mathcal{F}^0]} \Big)^2 |X_s|^{p-4}\\
			&\quad+ \sum_{k=2}^{2i-j-2} \binom{2i-j-2}{k}
			\Big( 1 + \sqrt{\mathbb{E}[|X_s|^2|\mathcal{F}^0]} \Big)^k |X_s|^{p-2-k} \Bigg)\\				
			&\leq |c(X_s, \mathcal{L}^1_{X_s})|^{2} |X_s|^{p-2} \left( 
			\frac{i}{L_3^2} (L_3^2|z|^2 + 2L_3|z|)^{i} 
			- \frac{|z|i}{L_3} (L_3^2|z|^2 + 2L_3|z|)^{i-1} 
			\right)\\
			&\quad+ {L_3^2}\Big(1+|X_{s}|+\sqrt{\mathbb{E}[|X_s|^2|\mathcal{F}^0]}\Big)^2\sum_{j=0}^i \binom{i}{j} 2^j |z|^{2i-j} 
			L_3^{2i-j-2} \Bigg( 
			(i - \frac{j}{2} - 1)(1 + \sqrt{\mathbb{E}[|X_s|^2|\mathcal{F}^0]})^2 \\
			&\quad \times|X_s|^{p-4} + \sum_{k=2}^{2i-j-2} \binom{2i-j-2}{k} (1 + \sqrt{\mathbb{E}[|X_s|^2|\mathcal{F}^0]})^k |X_s|^{p-2-k} 
			\Bigg) \\
			&=|c(X_s, \mathcal{L}^1_{X_s})|^{2} |X_s|^{p-2} 
			\frac{i|z|}{L_3}(1 + L_3|z|)(L_3^2|z|^2 + 2L_3|z|)^{i-1}\\
			&\quad+ \sum_{j=0}^i |z|^{2i-j} Q_p \big(p-2, 2i-j, |X_s|, 1+\sqrt{\mathbb{E}[|X_s|^2|\mathcal{F}^0]}\big),\\
			&\text{where} \quad 2i - j \geq 2 \quad \text{and} \quad 
		\end{align*}
		\begin{equation*}
			Q_q(m, n, x, y) = \sum_{\ell_1 \leq m, \, \ell_2 \leq n, \, \ell_1 + \ell_2 = q} c_{\ell_1 \ell_2} x^{\ell_1} y^{\ell_2}.
		\end{equation*}
		This, together with the identity
		\[
		\sum_{i=2}^{p/2} \binom{p/2}{i} i a^{i-1} = 
		\frac{p}{2} \left((1 + a)^{p/2 - 1} - 1\right),
		\]
		which holds for any $a \in \mathbb{R}$, yields,
		\begin{equation}\label{2.6}
			\begin{aligned}
				&\sum_{i=2}^{p/2} \binom{p/2}{i} |X_s|^{p-2i} 
				\left( |c(X_s, \mathcal{L}^1_{X_s})z|^2 + 2 \langle X_s, c(X_s, \mathcal{L}^1_{X_s})z \rangle \right)^i \\
				&\leq \frac{|z|}{L_3} \left( 1 + L_3 |z| \right) |c(X_s, \mathcal{L}^1_{X_s})|^2 |X_s|^{p-2} 
				\sum_{i=2}^{p/2} \binom{p/2}{i} i \left( L_3^2 |z|^2 + 2L_3 |z| \right)^{i-1} \\ 
				&\quad+ \sum_{i=2}^{p/2} \sum_{j=0}^i \binom{p/2}{i} |z|^{2i-j} Q_p(p - 2, 2i - j, |X_s|, 1 + \sqrt{\mathbb{E}[|X_s|^2|\mathcal{F}^0]})\\
				&= \frac{p}{2} |X_s|^{p-2} |c(X_s, \mathcal{L}^1_{X_s})|^2 \frac{|z|}{L_3} 
				\left( (1 + L_3 |z|)^{p-1} - L_3 |z| - 1 \right) \\ 
				&\quad+ \sum_{i=2}^{p/2} \sum_{j=0}^i \binom{p/2}{i} |z|^{2i-j} Q_p(p - 2, 2i - j, |X_s|, 1 + \sqrt{\mathbb{E}[|X_s|^2|\mathcal{F}^0]}). 
			\end{aligned}
		\end{equation}
		As a consequence of \eqref{2.5} and \eqref{2.6}, we conclude that for any $s \geq 0$,
		\begin{equation}\label{2.7}
			\begin{aligned}
				&|X_s + c(X_s, \mathcal{L}^1_{X_s})z|^p - |X_s|^p - p|X_s|^{p-2} \langle X_s, c(X_s, \mathcal{L}^1_{X_s})z \rangle \\
				&\leq \frac{p}{2} |X_s|^{p-2} |c(X_s, \mathcal{L}^1_{X_s})|^2 \Bigg( |z|^2 + \frac{|z|}{L_3} \left( (1 + L_3|z|)^{p-1} - L_3|z| - 1 \right) \Bigg) \\
				&\quad+ \sum_{i=2}^{p/2} \sum_{j=0}^i \binom{p/2}{i} |z|^{2i-j} Q_p(p - 2, 2i - j, |X_s|, 1 + \sqrt{\mathbb{E}[|X_s|^2|\mathcal{F}^0]})\\
				&= \frac{p}{2L_3} |X_s|^{p-2} |c(X_s, \mathcal{L}^1_{X_s})|^2 |z| \left( (1 + L_3|z|)^{p-1} - 1 \right) \\
				&\quad+ \sum_{i=2}^{p/2} \sum_{j=0}^i \binom{p/2}{i} |z|^{2i-j} Q_p(p - 2, 2i - j, |X_s|, 1 + \sqrt{\mathbb{E}[|X_s|^2|\mathcal{F}^0]}).
			\end{aligned}
		\end{equation}
		Under Assumption \textbf{(A4)}, the L\'evy measure satisfies 
		\[
		\int_{|z|>1} |z|^{2p_0} \nu({\rm{d}}z) < \infty,
		\]
		which ensures the convergence of the higher-order jump integrals in \eqref{2.5}. 
		Therefore, by substituting \eqref{2.7} into \eqref{2.4}, we obtain that for any $t \geq 0$,
		\begin{align*}
			&e^{-\lambda t} |X_t|^p \leq |x_0|^p +p \int_0^t e^{-\lambda s} |X_s|^{p-2} \Bigg[ 
			- \frac{\lambda}{p}|X_s|^2 + \langle X_s, b(X_s, \mathcal{L}^1_{X_s}) \rangle + \frac{p-1}{2} |\sigma(X_s, \mathcal{L}^1_{X_s})|^2 \\
			&+ \frac{p-1}{2} |\sigma^{0}(X_s, \mathcal{L}^1_{X_s})|^2+\frac{1}{2L_3} |c(X_s, \mathcal{L}^1_{X_s})|^2  \int_{\mathbb{R}^d_0} |z|\bigg( (1 + L_3|z|)^{p-1} - 1 \bigg) \nu({{\rm{d}}}z)\Bigg]{{\rm{d}}}s \\
			& +\int_{0}^{t}\int_{\mathbb{R}^d_0}e^{-\lambda s} \sum_{i=2}^{p/2} \sum_{j=0}^i \binom{p/2}{i} |z|^{2i-j} Q_p\bigg(p - 2, 2i - j, |X_s|, 1 + \sqrt{\mathbb{E}[|X_s|^2|\mathcal{F}^0]}\bigg) \nu({{\rm{d}}}z){{\rm{d}}}s \\
			&+ p \int_0^t e^{-\lambda s} |X_s|^{p-2} \bigg[\langle X_s, \sigma(X_s, \mathcal{L}^1_{X_s}) {{\rm{d}}}W_s \rangle+ \langle X_s, \sigma^{0}(X_s, \mathcal{L}^{1}_{X_s}) {{\rm{d}}}W^{0}_s \rangle \Bigg]  \\
			&+ \int_0^t \int_{\mathbb{R}^d_0} e^{-\lambda s} \Big( |X_{s-} + c(X_{s-}, \mathcal{L}^1_{X_s-})z|^p - |X_{s-}|^p \Big) \tilde{N}({{\rm{d}}}s, {{\rm{d}}}z).
		\end{align*}
		Now, for each $N > 0$, define the stopping time 
		\[
		\tau_N := \inf \{ t \geq 0 : |X_t| \geq N \}.
		\]
		By choosing $\lambda = \gamma_1 p$ and applying Assumption \textbf{(A6)}, Remark \ref{rem}, together with Lemma \ref{lemma}
		\[
		\mathcal{W}_2^2(\mathcal{L}^1_{X_s}, \delta_0) = \mathbb{E}[|X_s|^2 \mid \mathcal{F}^0],
		\]
		we obtain,
		\begin{equation}\label{2.8}
			\begin{aligned}
				&\mathbb{E} \left[ e^{-\gamma_1 p(t \wedge \tau_N)} |X_{t \wedge \tau_N}|^p \right] 
				\leq |x_0|^p + p\int_0^t e^{-\gamma_1 ps} \mathbb{E} \Big[ |X_s|^{p-2}\Big] \Big( \gamma_2 \mathcal{W}_2^2(\mathcal{L}^1_{X_s}, \delta_0) + \eta \Big) {{\rm{d}}}s\\
				&\quad+ \int_0^t \int_{\mathbb{R}^d_0} e^{-\gamma_1 p s}  \sum_{i=2}^{p/2} \sum_{j=0}^i \binom{p/2}{i} |z|^{2i - j} \mathbb{E} \left[ Q_p \left( p - 2, 2i - j, |X_s|, 1 + \sqrt{\mathbb{E}[|X_s|^2|\mathcal{F}^0]} \right) \right] \nu({{\rm{d}}}z) {{\rm{d}}}s \\
				&\leq |x_0|^p + p \int_0^t e^{-\gamma_1 p s} \mathbb{E} \left[ |X_s|^{p-2} \right]\big( \gamma_2 \mathbb{E} \left[ |X_s|^{2}|\mathcal{F}^0 \right] + \eta \big){{\rm{d}}}s\\
				&\quad+ \int_0^t \int_{\mathbb{R}^d_0} e^{-\gamma_1 p s}  \sum_{i=2}^{p/2} \sum_{j=0}^i \binom{p/2}{i} |z|^{2i - j} \mathbb{E} \left[ Q_p \left( p - 2, 2i - j, |X_s|, 1 + \sqrt{\mathbb{E}[|X_s|^2|\mathcal{F}^0]} \right) \right] \nu({{\rm{d}}}z) {{\rm{d}}}s.
			\end{aligned}
		\end{equation}
		Next, using the inequality 
		\[
		e^{-\gamma_1 p (t \wedge \tau_N)} \geq e^{-\gamma_1 p t},
		\]
		and the induction hypothesis \eqref{2.3}, there exists a positive constant $C(T, p)$, independent of $N$, such that
		
		\begin{equation}\label{2.9}
			\sup_{t \in [0,T]} \mathbb{E} \left[ |X_{t \wedge \tau_N}|^p \right] \leq C(T, p). 
		\end{equation}
		This yields
		\[
		\sup_{t \in [0,T]} \mathbb{P}(\tau_N < t) \leq \frac{C(T, p)}{N^p},
		\]
		which implies that \(\tau_N \to \infty\) almost surely as \(N \to \infty\).  
		Now, letting \(N \uparrow \infty\) and applying Fatou's lemma to the left-hand side of \eqref{2.9}, we obtain
		\[
		\sup_{t \in [0,T]} \mathbb{E}\left[ |X_t|^p \right] \leq C(T, p).
		\]
		Thus, by the principle of induction, we establish \eqref{2.2}.\\
		\textbf{Step 3:} We next aim to prove \eqref{2.1} for every even natural number $p \in [2, p_0]$. \\
		First, by applying \eqref{2.4} with $\lambda = 2\gamma$ and $p = 2$, and using Assumption \textbf{(A6)}, we obtain
		\begin{align*}
			&e^{-2\gamma t} |X_t|^2 
			\leq |x_0|^2 + 2 \int_0^t e^{-2\gamma s} \left( -\gamma_2 |X_s|^2 + \gamma_2 \mathcal{W}_2^2(\mathcal{L}^1_{X_s}, \delta_0) + \eta \right) {{\rm{d}}}s\\
			&\quad+ 2 \int_0^t e^{-2\gamma s} \bigg[\langle X_s, \sigma(X_s, \mathcal{L}^1_{X_s}) {{\rm{d}}}W_s+\sigma^{0}(X_s, \mathcal{L}^1_{X_s}) {{\rm{d}}}W_s^{0} \rangle\bigg] \\
			&\quad+ \int_0^t \int_{\mathbb{R}^d_0} e^{-2\gamma s}  \Big( |X_{s-} + c(X_{s-}, \mathcal{L}^1_{X_s-})z|^2 - |X_{s-}|^2 \Big) \tilde{N}({{\rm{d}}}s, {{\rm{d}}}z).
		\end{align*}
		Due to Lemma \ref{lemma}, we have
		\(
		\mathcal{W}_2^2\big(\mathcal{L}^1_{X_s}, \delta_0\big) = \mathbb{E}\big[\,|X_s|^2 \,\big|\, \mathcal{F}^0\big],
		\)
		and by applying the estimate \eqref{2.2}, we deduce
		\[
		\mathbb{E}\!\left[e^{-2\gamma t} |X_t|^2\right] 
		\le |x_0|^2 + 2 \int_0^t e^{-2\gamma s} \mathbb{E}\!\left[ -\gamma_2 |X_s|^{2} + \gamma_2 \mathbb{E}[|X_s|^2 \mid \mathcal{F}^0] + \eta \right] \mathrm{d}s 
		\le |x_0|^2 + 2\eta \int_0^t e^{-2\gamma s} \,\mathrm{d}s.
		\]
		This yields
		\[
		\mathbb{E}[|X_t|^2] \le 
		\begin{cases}
			\big( |x_0|^2 + \tfrac{\eta}{\gamma} \big) e^{2\gamma t} - \tfrac{\eta}{\gamma} & \gamma \neq 0, \\[0.4em]
			|x_0|^2 + 2\eta t & \gamma = 0,
		\end{cases}
		\]
		so that \eqref{2.1} holds in the case \(p=2\).  
		Assume now, as an induction hypothesis, \eqref{2.1} is valid for every even integer \(q \in [2,\, p-2]\), that is,
		\begin{equation}\label{2.10}
			\mathbb{E}[|X_t|^q] \le 
			\begin{cases}
				C_q \big( 1 + e^{\gamma q t} \big) & \gamma \neq 0, \\[0.4em]
				C_q (1+t)^{q/2} & \gamma = 0,\ \big(q = 2\ \text{or}\ (\gamma_2>0,\ q\in(2, p-2])\big), \\[0.4em]
				C_q (1+t)^q & \gamma = 0,\ \gamma_2 = 0,\ q \in (2, p-2].
			\end{cases}
		\end{equation}
		We prove that \eqref{2.1} also holds for the even integer \(p\).  
		The argument relies on \eqref{2.8}, the inductive assumption \eqref{2.10}, and Assumption~\textbf{(A4)}.  
		
		\paragraph{Case \(\gamma \neq 0\).}
		We have
		\[
		\mathbb{E} \!\left[ e^{-\gamma_1 p (t \wedge \tau_N)} |X_{t \wedge \tau_N}|^p \right] 
		\le |x_0|^p + C \int_0^t e^{-\gamma_1 p s} \big( 1 + e^{\gamma p s} \big) \mathrm{d}s
		= |x_0|^p + C \int_0^t \big( e^{-\gamma_1 p s} + e^{\gamma_2 p s} \big) \mathrm{d}s.
		\]
		Since $\tau_N \to \infty$ almost surely as $N \to \infty$, Fatou's lemma gives
		\[
		\mathbb{E} \!\left[ e^{-\gamma_1 p t} |X_t|^p \right] 
		\le |x_0|^p + C \int_0^t \big( e^{-\gamma_1 p s} + e^{\gamma_2 p s} \big) \mathrm{d}s.
		\]
		If $\gamma_1 = 0$, then
		\[
		\mathbb{E}[|X_t|^p] \le |x_0|^p + C\big(1 + t + e^{\gamma_2 p t}\big) \le C\big(1 + e^{\gamma p t}\big).
		\]
		If $\gamma_1 \neq 0$, we obtain
		\[
		\mathbb{E}[|X_t|^p] \le |x_0|^p e^{\gamma_1 p t} + \frac{C}{p\gamma_1}\big(e^{\gamma_1 p t} - 1\big) + \frac{C}{p\gamma_2} e^{\gamma p t} \le C\big(1 + e^{\gamma p t}\big).
		\]
		
		\paragraph{Case \(\gamma = 0\).}
		If $\gamma_2 > 0$, then
		\[
		\mathbb{E} \!\left[ e^{-\gamma_1 p (t \wedge \tau_N)} |X_{t \wedge \tau_N}|^p \right] 
		\le |x_0|^p + C \int_0^t e^{-\gamma_1 p s} (1 + s)^{p/2} \mathrm{d}s 
		\le |x_0|^p + C (1 + t)^{p/2} \!\! \int_0^t e^{-\gamma_1 p s} \mathrm{d}s.
		\]
		Letting $N \to \infty$ and using $-\gamma_1 = \gamma_2$, we find
		\[
		\mathbb{E}\!\left[ e^{\gamma_2 p t} |X_t|^p \right] 
		\le |x_0|^p + \frac{C}{\gamma_2 p} (1 + t)^{p/2} e^{\gamma_2 p t},
		\]
		hence $\mathbb{E}[|X_t|^p] \le C (1 + t)^{p/2}$.\\
		If $\gamma_2 = \gamma_1 = 0$, we have
		\[
		\mathbb{E}[|X_{t \wedge \tau_N}|^p] \le |x_0|^p + C \int_0^t (1 + s)^{p-1} \mathrm{d}s \le C (1 + t)^p.
		\]
		Letting $N \to \infty$ yields $\mathbb{E}[|X_t|^p] \le C (1 + t)^p$.
		
		Therefore, \eqref{2.1} holds for the given even $p$.  
		By induction, it follows that \eqref{2.1} is valid for all even integers $p \in [2,\, p_0]$, and by H\"older's inequality, for all real $p \in [2,\, p_0]$.  
		This completes the proof.		
	\end{proof}
	\section{Propagation of chaos}	
	For each \( N \in \mathbb{N} \), let \((W^i, Z^i)\), \(i = 1, \ldots, N\), be independent copies of the pair \((W, Z)\). Denote by \(N^i({\rm{d}}t, {\rm{d}}z)\) the Poisson random measure corresponding to the jumps of the L\'evy process \(Z^i\) with intensity measure \(\nu({\rm{d}}z) {\rm{d}}t\), and define the compensated Poisson random measure by
	\[
	\widetilde{N}^i({\rm{d}}t, {\rm{d}}z) := N^i({\rm{d}}t, {\rm{d}}z) - \nu({\rm{d}}z) {\rm{d}}t.
	\]
	Then, the L\'evy-It\^o decomposition of \(Z^i\) is given by
	\[
	Z_t^i = \int_0^t \int_{\mathbb{R}_0^d} z \, \widetilde{N}^i({\rm{d}}s, {\rm{d}}z), \quad t \geq 0.
	\]
	We now consider the system of non-interacting particles associated with the  McKean-Vlasov SDE driven by L\'evy processes with common noise \eqref{eq:MKV-SDE}, where the state \(X^i = (X_t^i)_{t \geq 0}\) of the \(i\)-th particle is defined by
	\begin{equation}
		\begin{split}
			X_t^i&= x_0 + \int_0^t b(X_s^i, \mathcal{L}^{1}_{X_s^i}) {{\rm{d}}}s + \int_0^t \sigma(X_s^i, \mathcal{L}^{1}_{X_s^i}) {{\rm{d}}}W_s^i+ \int_0^t \sigma^{0}(X_s^i, \mathcal{L}^{1}_{X_s^i}) {{\rm{d}}}W^{0}_s \\
			&+ \int_0^t \int_{\mathbb{R}_0^d} c\left(X_{s-}^i, \mathcal{L}^{1}_{X_{s-}^i}\right) z \widetilde{N}^i({{\rm{d}}}s, {{\rm{d}}}z),
		\end{split}
	\end{equation}
	for any $t \geq 0$ and $i \in \{1, \ldots, N\}$.\\
	For $\boldsymbol{X}^N := (X^1, X^2, \ldots, X^N),   \boldsymbol{Y}^N := (Y^1, Y^2, \ldots, Y^N) \in \mathbb{R}^{dN}$, we have
	\[
	\mathcal{W}_2^2(\mu^{\boldsymbol{X}^N}, \delta_0) = \frac{1}{N} \sum_{i=1}^N |X^i|^2.
	\]
	Here, the empirical measure is defined by
	\begin{equation}\label{ppp}
		\mu^{\boldsymbol{X}^N}({\rm{d}}x) := \frac{1}{N} \sum_{i=1}^N \delta_{X^i}({\rm{d}}x),
	\end{equation}
	where \(\delta_X\) denotes the Dirac measure concentrated at \(X\). 
	Moreover, a standard estimate for the Wasserstein distance between two empirical measures \(\mu^{\boldsymbol{X}^N}\) and \(\mu^{\boldsymbol{Y}^N}\) is given by
	\begin{equation}\label{eee}
		\mathcal{W}_2^2(\mu^{\boldsymbol{X}^N}, \mu^{\boldsymbol{Y}^N}) \leq \frac{1}{N} \sum_{i=1}^N |X^i - Y^i|^2 = \frac{1}{N} \left| \boldsymbol{X}^N - \boldsymbol{Y}^N \right|^2,
	\end{equation}
	(see inequality (1.24) in \cite{Carmona2016}).
	The true measure \(\mathcal{L}^1_{X_t}\) at each time \(t\) is approximated by the empirical measure
	\begin{equation}\label{qqq}
		\mu_t^{\boldsymbol{X}^N}({\rm{d}}x) := \frac{1}{N} \sum_{i=1}^N \delta_{X_t^{i,N}}({\rm{d}}x),
	\end{equation}
	where \(\boldsymbol{X}^N = (\boldsymbol{X}_t^N)_{t \geq 0} = (X_t^{1,N}, \ldots, X_t^{N,N})_{t \geq 0}^\top\) denotes the system of interacting particles, which is the solution to an \(\mathbb{R}^{dN}\)-valued SDE driven by L\'evy processes with common noise. The components \(X^{i,N} = (X_t^{i,N})_{t \geq 0}\) correspond to the state of the \(i\)-th particle. Then, we have
	\begin{equation}\label{3.2}
		\begin{split}
			X_t^{i,N} &= x_0 + \int_0^t b(X_s^{i,N}, \mu_s^{\boldsymbol{X}^N}) {{\rm{d}}}s + \int_0^t \sigma(X_s^{i,N}, \mu_s^{\boldsymbol{X}^N}) {{\rm{d}}}W_s^i+ \int_0^t \sigma^{0}(X_s^{i,N}, \mu_s^{\boldsymbol{X}^N}) {{\rm{d}}}W^{0}_s \\
			&\quad+ \int_0^t \int_{\mathbb{R}_0^d} c\left(X_{s-}^{i,N}, \mu_{s-}^{\boldsymbol{X}^N}\right) z \widetilde{N}^i({\rm{d}}s,{\rm{d}}z),
		\end{split}
	\end{equation}
	for any $t \geq 0$ and $i \in \{1, \ldots, N\}$.
	Note that the interacting particle system \(\boldsymbol{X}^N = (X^{i,N})_{i=1}^N{}^\top\) can be regarded as a SDE driven by L\'evy processes with common noise and random coefficients taking values in \(\mathbb{R}^{dN}\). 
	Therefore, under assumptions \textbf{(A1)}, \textbf{(A2)}, and \textbf{(A3)} and \(L_1 = L_2 > 0\), there exists a unique c$\grave{a}$dl$\grave{a}$g solution to \eqref{3.2} satisfying
	\[
	\max_{i \in \{1, \dots, N\}} \sup_{t \in [0,T]} \mathbb{E} \big[ |X^{i,N}_t|^2 \big] \leq K,
	\]
	for any \(N \in \mathbb{N}\), where the constant \(K > 0\) is independent of \(N\).
	
	\begin{prp}\label{prop:3.1}
		(\cite[Proposition 3.1]{Tran2025})  Let assumptions \textup{\textbf{(A1)-(A7)}} hold,  for each \(p \in [2, p_0]\), there exists a constant \(C_p > 0\) such that, for all \(t \ge 0\),
		\[
		\max_{1 \le i \le N} \mathbb{E}\!\left[ |X_t^{i,N}|^p \right] \le 
		\begin{cases}
			C_p\big(1 + e^{\gamma p t}\big) & \gamma \neq 0, \\[0.3em]
			C_p(1 + t)^{p/2} & \gamma = 0,\ p = 2 \ \text{or} \ (\gamma = 0,\ \gamma_2 > 0,\ p \in (2, p_0]), \\[0.3em]
			C_p(1 + t)^p & \gamma = 0,\ \gamma_2 = 0,\ p \in (2, p_0],
		\end{cases}
		\]
		where \(\gamma := \gamma_1 + \gamma_2\).  
		In the special case \(\gamma < 0\), one further has
		\[
		\max_{1 \le i \le N} \ \sup_{t \ge 0} \ \mathbb{E}\!\left[\,|X_t^{i,N}|^p\,\right] \le 2 C_p.
		\]
	\end{prp}
	\begin{proof}
		The argument follows the same reasoning as in the proof of Proposition~\ref{prp:2.3} and is therefore omitted here.
	\end{proof}

	\begin{prp}\label{prop:3.2}
		Define the rate function
		\begin{equation}\label{3.3}
			\varphi(N) = 
			\begin{cases} 
				N^{-1/2} & \text{for } d < 4, \\
				N^{-1/2}\ln N & \text{for } d = 4,\\
				N^{-2/d} & \text{for } d > 4.
			\end{cases}
		\end{equation}
		Under the conditions of Proposition \ref{prop:3.1} and Assumption \textbf{(A2)} with $L_1 \in \mathbb{R}$, $L_2 \geq 0$, there exists a constant $C_T > 0$, independent of $N$, such that for all $N \in \mathbb{N}$,
		\begin{equation}\label{3.4}
			\max_{i \in \{1, \ldots, N\}} \sup_{t \in [0,T]} \mathbb{E}\left[ |X_t^i - X_t^{i,N}|^2 \right] \leq C_T \varphi(N).
		\end{equation}
		Furthermore, if $L_1 + L_2 < 0$ and $\gamma < 0$, then the constant can be made independent of time, i.e., there exists $C > 0$, independent of $N$ and $T$, such that
		\begin{equation}\label{3.5}
			\max_{i \in \{1, \ldots, N\}} \sup_{t \geq 0} \mathbb{E}\left[ |X_t^i - X_t^{i,N}|^2 \right] \leq C \varphi(N).
		\end{equation}
	\end{prp}
	
	\begin{proof}
		The proof of the theorem follows the classical approach presented in \cite{Tran2025}. It is important to emphasize that the main novelty of our result is the extension of the estimate to the setting with common noise.
		Observe that for any \( t \geq 0 \),
		\begin{align*}
			&X_t^i - X_t^{i,N} = \int_0^t \left( b(X_s^i, \mathcal{L}^{1}_{X_s^i}) - b(X_s^{i,N}, \mu_s^{\boldsymbol{X}^N}) \right) {{\rm{d}}}s + \int_0^t \left( \sigma(X_s^i, \mathcal{L}^{1}_{X_s^i}) - \sigma(X_s^{i,N}, \mu_s^{\boldsymbol{X}^N}) \right) {\rm{d}}W_s^i \\
			&+ \int_0^t \left( \sigma^{0}(X_s^i, \mathcal{L}^{1}_{X_s^i}) - \sigma^{0}(X_s^{i,N}, \mu_s^{\boldsymbol{X}^N}) \right) {\rm{d}}W_s^{0} + \int_0^t \int_{\mathbb{R}_0^d} \left( c(X_{s-}^i, \mathcal{L}^{1}_{X_{s-}^i}) - c(X_{s-}^{i,N}, \mu_{s-}^{\boldsymbol{X}^N}) \right) z \widetilde{N}^i({\rm{d}}s, {\rm{d}}z).
		\end{align*}
		Next, for \(\lambda \in \mathbb{R}\), by applying It\^o's formula and invoking Assumption \textbf{(A2)}, \(L_1 \in \mathbb{R}\), and \(L_2 \geq 0\), we deduce that for all \(t \geq 0\),
		\begin{equation}\label{nnn}
			\begin{aligned}
				&e^{-\lambda t} |X_t^i - X_t^{i,N}|^2 = \int_0^t e^{-\lambda s} \Big[ -\lambda |X_s^i - X_s^{i,N}|^2 + 2 \langle X_s^i - X_s^{i,N}, b(X_s^i, \mathcal{L}^{1}_{X_s^i}) - b(X_s^{i,N}, \mu_s^{\boldsymbol{X}^N}) \rangle \\
				&\quad+ |\sigma(X_s^i, \mathcal{L}^{1}_{X_s^i}) - \sigma(X_s^{i,N}, \mu_s^{\boldsymbol{X}^N})|^2 + |\sigma^{0}(X_s^i, \mathcal{L}^{1}_{X_s^i}) - \sigma^{0}(X_s^{i,N}, \mu_s^{\boldsymbol{X}^N})|^2 \Big] {\rm{d}}s \\
				&\quad+ 2 \int_0^t e^{-\lambda s} \bigg[ \langle X_s^i - X_s^{i,N}, (\sigma(X_s^i, \mathcal{L}^{1}_{X_s^i}) - \sigma(X_s^{i,N}, \mu_s^{\boldsymbol{X}^N})) {\rm{d}}W_s^i \rangle \\
				&\quad+ \langle X_s^i - X_s^{i,N}, (\sigma^{0}(X_s^i, \mathcal{L}^{1}_{X_s^i}) - \sigma^{0}(X_s^{i,N}, \mu_s^{\boldsymbol{X}^N})) {\rm{d}}W_s^{0} \rangle \bigg] \\
				&\quad+ \int_0^t \int_{\mathbb{R}_0^d} e^{-\lambda s} \left( |X_{s-}^i - X_{s-}^{i,N} + \big(c(X_{s-}^i, \mathcal{L}^{1}_{X_{s-}^i}) - c(X_{s-}^{i,N}, \mu_{s-}^{\boldsymbol{X}^N})\big) z|^2 - |X_{s-}^i - X_{s-}^{i,N}|^2 \right) \widetilde{N}^i({\rm{d}}s, {\rm{d}}z) \\
				&\quad+ \int_0^t \int_{\mathbb{R}_0^d} e^{-\lambda s} |(c(X_{s-}^i, \mathcal{L}^{1}_{X_{s-}^i}) - c(X_{s-}^{i,N}, \mu_{s-}^{\boldsymbol{X}^N})) z|^2 \nu({\rm{d}}z) {\rm{d}}s \\
				&\leq \int_{0}^{t} e^{-\lambda s} \left( -\lambda |X_s^{i} - X_s^{i,N}|^2 + 2\left\langle X_s^{i} - X_s^{i,N}, b(X_s^{i}, \mathcal{L}^{1}_{X_{s}^i}) - b(X_s^{i,N}, \mu_s^{\boldsymbol{X}^N}) \right\rangle \right) {\rm{d}}s \\
				&\quad+ \int_{0}^{t} e^{-\lambda s} \left( |\sigma(X_s^{i}, \mathcal{L}^{1}_{X_{s}^i}) - \sigma(X_s^{i,N}, \mu_s^{\boldsymbol{X}^N})|^2 + |\sigma^{0}(X_s^{i}, \mathcal{L}^{1}_{X_{s}^i}) - \sigma^{0}(X_s^{i,N}, \mu_s^{\boldsymbol{X}^N})|^2 \right) {\rm{d}}s \\
				&\quad+ \int_{0}^{t} e^{-\lambda s} |c(X_{s-}^{i}, \mathcal{L}^{1}_{X_{s-}^i}) - c(X_{s-}^{i,N}, \mu_{s-}^{\boldsymbol{X}^N})|^2 \int_{\mathbb{R}_0^d} |z|^2 \nu({\rm{d}}z) {\rm{d}}s \\
				&\quad+ 2 \int_{0}^{t} e^{-\lambda s} \bigg[ \left\langle X_s^{i} - X_s^{i,N}, \left( \sigma(X_s^{i}, \mathcal{L}^{1}_{X_{s}^i}) - \sigma(X_s^{i,N}, \mu_s^{\boldsymbol{X}^N}) \right) {\rm{d}}W_s^i \right\rangle \\
				&\quad+ \left\langle X_s^{i} - X_s^{i,N}, \left( \sigma^{0}(X_s^{i}, \mathcal{L}^{1}_{X_{s}^i}) - \sigma^{0}(X_s^{i,N}, \mu_s^{\boldsymbol{X}^N}) \right) {\rm{d}}W_s^{0} \right\rangle \bigg] \\
				&\quad+ \int_{0}^{t} \int_{\mathbb{R}_0^d} e^{-\lambda s} \left( |X_{s-}^{i} - X_{s-}^{i,N} + \left( c(X_{s-}^{i}, \mathcal{L}^{1}_{X_{s-}^{i}}) - c(X_{s-}^{i,N}, \mu_{s-}^{\boldsymbol{X}^N}) \right) z|^2 - |X_{s-}^{i} - X_{s-}^{i,N}|^2 \right) \widetilde{N}^i({\rm{d}}s, {\rm{d}}z) \\
				&\leq \int_{0}^{t} e^{-\lambda s} \left( -\lambda |X_s^{i} - X_s^{i,N}|^2 + L_1 |X_s^{i} - X_s^{i,N}|^2 + L_2 \mathcal{W}_2^2(\mathcal{L}^{1}_{X_{s}^i}, \mu_s^{\boldsymbol{X}^N}) \right) {\rm{d}}s \\
				&\quad+ 2 \int_{0}^{t} e^{-\lambda s} \bigg[ \left\langle X_s^{i} - X_s^{i,N}, \left( \sigma(X_s^{i}, \mathcal{L}^{1}_{X_{s}^i}) - \sigma(X_s^{i,N}, \mu_s^{\boldsymbol{X}^N}) \right) {\rm{d}}W_s^i \right\rangle \\
				&\quad+ \left\langle X_s^{i} - X_s^{i,N}, \left( \sigma^{0}(X_s^{i}, \mathcal{L}^{1}_{X_{s}^i}) - \sigma^{0}(X_s^{i,N}, \mu_s^{\boldsymbol{X}^N}) \right) {\rm{d}}W_s^{0} \right\rangle \bigg] \\
				&\quad+ \int_{0}^{t} \int_{\mathbb{R}_0^d} e^{-\lambda s} \left( |X_{s-}^{i} - X_{s-}^{i,N} + \left( c(X_{s-}^{i}, \mathcal{L}^{1}_{X_{s-}^{i}}) - c(X_{s-}^{i,N}, \mu_{s-}^{\boldsymbol{X}^N}) \right) z|^2 - |X_{s-}^{i} - X_{s-}^{i,N}|^2 \right) \widetilde{N}^i({\rm{d}}s, {\rm{d}}z).
			\end{aligned}
		\end{equation}
		By taking the expectation to \eqref{nnn}, we have for any \(\varepsilon > 0\),
		\begin{equation}
			\begin{split}\label{3.6}
				&e^{-\lambda t} \mathbb{E}\left[ |X_t^i - X_t^{i,N}|^2 \right]\\
				&\leq \int_0^t e^{-\lambda s}\bigg((-\lambda+L_1)\mathbb{E}\left[ |X_s^i - X_s^{i,N}|^2 \right]+L_2\mathbb{E}\left[ \mathcal{W}_2^2(\mathcal{L}^{1}_{X_s^{i}}, \mu_s^{\boldsymbol{X}^N}) \right]\bigg){\rm{d}}s\\
				&\leq \int_0^t e^{-\lambda s} \left( (-\lambda + L_1 + L_2 + \varepsilon) \mathbb{E}\left[ |X_s^i - X_s^{i,N}|^2 \right] + L_2 \left(1 + \frac{L_2}{\varepsilon}\right) \mathbb{E}\left[ \mathcal{W}_2^2(\mathcal{L}^{1}_{X_s^{i,N}}, \mu_s^{\boldsymbol{X}^N}) \right] \right){\rm{d}}s.
			\end{split}
		\end{equation}
		Next, according to \eqref{ppp}, \eqref{eee},  \eqref{qqq} and \(	\mathbb{E}\left[\mathcal{W}_2^2\big(\mathcal{L}^{1}_{X_s^i}, \mathcal{L}^{1}_{X_s^{i,N}}\big)\right] \leq \mathbb{E}\left[|X_s^i - X_s^{i,N}|^2\right]\), let \(|X_s^j - X_s^{j,N}|^2:=\max_{i \in \{1, \ldots, N\}}|X_s^i - X_s^{i,N}|^2\) we have
		\begin{equation}\label{ooo}
			\begin{split}
				\mathbb{E}\left[ \mathcal{W}_2^2(\mathcal{L}^{1}_{X_s^{i,N}}, \mu_s^{\boldsymbol{X}^N}) \right] 	
				&\leq \mathbb{E}[ \mathcal{W}_2(\mathcal{L}^{1}_{X_s^{i,N}}, \mathcal{L}^{1}_{X_s^{i}})+\mathcal{W}_2(\mathcal{L}^{1}_{X_s^{i}},\mu^{\boldsymbol{X}^N})+\mathcal{W}_2(\mu^{\boldsymbol{X}^N},\mu_s^{\boldsymbol{X}^N})]^2\\
				&\leq 3\mathbb{E}[ \mathcal{W}^2_2(\mathcal{L}^{1}_{X_s^{i,N}}, \mathcal{L}^{1}_{X_s^{i}})+\mathcal{W}^2_2(\mathcal{L}^{1}_{X_s^{i}},\mu^{\boldsymbol{X}^N})+\mathcal{W}^2_2(\mu^{\boldsymbol{X}^N},\mu_s^{\boldsymbol{X}^N})]\\
				&\leq 3\mathbb{E}[ \mathbb{E}\left[|X_s^i - X_s^{i,N}|^2\right]+\mathcal{W}^2_2(\mathcal{L}^{1}_{X_s^{i}},\mu^{\boldsymbol{X}^N})+\frac{1}{N} \sum_{i=1}^N |X_s^{i} - X_s^{i,N}|^2]\\
				&\leq 3\mathbb{E}[ \mathbb{E}\left[|X_s^i - X_s^{i,N}|^2\right]+\mathcal{W}^2_2(\mathcal{L}^{1}_{X_s^{i}},\mu^{\boldsymbol{X}^N})+ |X_s^{j} - X_s^{j,N}|^2]\\
				&\leq 6\mathbb{E}[  |X_s^{j} - X_s^{j,N}|^2]+3\mathbb{E}[\mathcal{W}^2_2(\mathcal{L}^{1}_{X_s^{i}},\mu^{\boldsymbol{X}^N})].
			\end{split}
		\end{equation}
		Moreover, by Proposition \ref{prop:3.1}, for any \( p \in [2, p_0] \), we have
		\[
		\max_{i \in \{1, \dots, N\}} \sup_{t \in [0,T]} \mathbb{E} \left[ \left| X^{i,N}_{t} \right|^p \right] \leq C_T,
		\]
		for some constant \( C_T > 0 \), together with Carmona and Delarue~\cite[Theorem 5.8]{Carmona2018} deduce that
		\begin{equation}\label{ddd}
		\mathbb{E}[\mathcal{W}^2_2(\mathcal{L}^{1}_{X_s^{i}},\mu^{\boldsymbol{X}^N})]\leq C \varphi(N).
		\end{equation}
		Combining \eqref{ooo} and \eqref{ddd}, we have
		\begin{equation}\label{uuu}
			\begin{aligned}
				\mathbb{E}\left[ \mathcal{W}_2^2(\mathcal{L}^{1}_{X_s^{i,N}}, \mu_s^{\boldsymbol{X}^N}) \right] 	&\leq 6\mathbb{E}[  |X_s^{j} - X_s^{j,N}|^2]+3\mathbb{E}[\mathcal{W}^2_2(\mathcal{L}^{1}_{X_s^{i}},\mu^{\boldsymbol{X}^N})]\\
		     	&\leq 6\mathbb{E}[  |X_s^{j} - X_s^{j,N}|^2]+3 C \varphi(N).
			\end{aligned}
		\end{equation}
		Consequently, it suffices to choose \( \lambda = L_1 + L_2 + \epsilon \) in \eqref{3.6} to conclude that
		\begin{equation}\label{bbb}
			\begin{split}
				&e^{-\lambda t} \mathbb{E}\left[ |X_t^i - X_t^{i,N}|^2 \right]\\
				&\leq \int_0^t e^{-\lambda s} \left( L_2 \left(1 + \frac{L_2}{\varepsilon}\right) \mathbb{E}\left[ \mathcal{W}_2^2(\mathcal{L}^{1}_{X_s^{i,N}}, \mu_s^{\boldsymbol{X}^N}) \right] \right){{\rm{d}}}s\\
				&\leq L_2 \left(1 + \frac{L_2}{\varepsilon}\right)\int_0^t e^{-\lambda s} \left( 6\mathbb{E}[  |X_s^{j} - X_s^{j,N}|^2]+3 C \varphi(N) \right){{{\rm{d}}}}s\\
				&=6L_2 \left(1 + \frac{L_2}{\varepsilon}\right)\int_0^t e^{-\lambda s}\mathbb{E}[|X_s^j-X_s^{j,N}|^2]{{{\rm{d}}}}s+3 C\varphi(N)L_2 \left(1 + \frac{L_2}{\varepsilon}\right)\int_0^t e^{-\lambda s}{{\rm{d}}}s.	
			\end{split}
		\end{equation}
	  Thus, we have 
		\[e^{-\lambda t} \mathbb{E}\left[ |X_t^j - X_t^{j,N}|^2 \right]\leq6L_2 \left(1 + \frac{L_2}{\varepsilon}\right)\int_0^t e^{-\lambda s}\mathbb{E}[|X_s^j-X_s^{j,N}|^2]{{{\rm{d}}}}s+3 C\varphi(N)L_2 \left(1 + \frac{L_2}{\varepsilon}\right)\int_0^t e^{-\lambda s}{{\rm{d}}}s. \]
		Then using Gr\"onwall's inequality, we have 
		\[
		\begin{split}
			\mathbb{E}\left[ |X_t^j - X_t^{j,N}|^2\right]
			&\leq \frac{3C\varphi(N)L_2\left(1 + \frac{L_2}{\varepsilon}\right)}{\lambda} e^{\int_0^t 6L_2\left(1 + \frac{L_2}{\varepsilon}\right) e^{\lambda (t-s)}{{\rm{d}}}s}\\
			&= \frac{3C\varphi(N)L_2\left(1 + \frac{L_2}{\varepsilon}\right)}{\lambda} 
			e^{-6L_2\left(1 + \frac{L_2}{\varepsilon}\right)(1-e^{\lambda t})/\lambda}\\
			&=C_T\varphi(N).\\
		\end{split}
		\]
		When \( L_1 + L_2 < 0 \) and \( \gamma = \gamma_1 + \gamma_2 < 0 \), one can choose \( \epsilon > 0 \) sufficiently small so that \( \lambda < 0 \). Hence, we have
		\[
		\begin{split}
			\mathbb{E}\left[ |X_t^j - X_t^{j,N}|^2\right]
			&\leq \frac{3C\varphi(N)L_2\left(1 + \frac{L_2}{\varepsilon}\right)}{\lambda} 
			e^{-6L_2\left(1 + \frac{L_2}{\varepsilon}\right)(1-e^{\lambda t})/\lambda}\\
			&\leq \frac{3C\varphi(N)L_2\left(1 + \frac{L_2}{\varepsilon}\right)}{\lambda}e^{-6L_2\left(1 + \frac{L_2}{\varepsilon}\right)/\lambda}\\
			&=C\varphi(N),
		\end{split}
		\]
		for any \( t \in [0, T] \), where the positive constant \( C \) does not depend on time.
	This completes the proof.
	\end{proof}
	\section{Tamed-Adaptive Euler-Maruyama (TAEM) Method}
	\subsection{Discrete-time and Continuous-time TAEM Method}
	Set
	\[
	t_0 := 0, \qquad X_{t_0} := X_0,
	\]
	where $X_0$ is the given initial random variable in \eqref{eq:MKV-SDE}.
	We now consider the tamed-adaptive EM scheme with step
	\begin{equation}\label{vvv}
		\Delta_{t_k}:=\frac{h}{1+\beta\,|b(X_{t_k},\mathcal L^1(X_{t_k}))|},\qquad  \beta>0,\ h\in(0,1],\qquad i.e.\qquad  t_{k+1}=t_k + \Delta_{t_k},
	\end{equation}
	and tamed drift
	\begin{equation}\label{bb}
	b_{\Delta_{t_k}}(x,\mu):=\frac{b(x,\mu)}{1+\Delta_{t_k}^{\alpha}|b(x,\mu)|},\qquad 0<\alpha<\frac{1}{2}.
	\end{equation}
	The TAEM discrete approximation is then defined by
	\begin{equation}\label{eq:TAEM-discrete}
		\begin{aligned}
			X_{t_{k+1}} &= X_{t_k} + b_{\Delta_{t_k}}\bigl(X_{t_k},\mathcal{L}^1(X_{t_k})\bigr)\,\Delta_{t_k}  + \sigma\bigl(X_{t_k},\mathcal{L}^1(X_{t_k})\bigr)\,\Delta W_{t_k} \\
			&\quad + c\bigl(X_{t_k},\mathcal{L}^1(X_{t_k})\bigr)\,\Delta Z_{t_k} + \sigma^0\bigl(X_{t_k},\mathcal{L}^1(X_{t_k})\bigr)\,\Delta W^0_{t_k},
		\end{aligned}
	\end{equation}
	where $\Delta W_{t_k} = W(t_{k+1})-W(t_k)$, $\Delta Z_{t_k}=Z(t_{k+1})-Z(t_k)$, and $\Delta W^0_{t_k}=W^0(t_{k+1})-W^0(t_k)$. The iteration stops when $t_k\ge T$. Note that at each time $t_k$,  the random variables \(\Delta_{t_k}\) and all coefficients are \(\mathcal F_{t_k}\)-measurable.
	
	The piecewise constant interpolant is defined by
	\begin{equation}\label{eq:step-interpolant}
		\overline{X}^\Delta(t) := X_{t_k}, \qquad t\in[t_k,t_{k+1}),
	\end{equation}
	and the continuous-time TAEM approximation is
	\begin{equation}\label{eq:TAEM-cont-expanded}
		\begin{aligned}
			X^\Delta(t) &= X_0 
			+ \int_0^t b_{\Delta_{t_s}}\bigl(\overline{X}^\Delta(s),\mathcal{L}^1(\overline{X}^\Delta(s))\bigr)\, \mathrm{d}s
			+ \int_0^t \sigma\bigl(\overline{X}^\Delta(s),\mathcal{L}^1(\overline{X}^\Delta(s))\bigr)\, \mathrm{d}W(s) \\
			&\quad + \int_0^t c\bigl(\overline{X}^\Delta(s),\mathcal{L}^1(\overline{X}^\Delta(s))\bigr)\, \mathrm{d}Z(s) 
			+ \int_0^t \sigma^0\bigl(\overline{X}^\Delta(s),\mathcal{L}^1(\overline{X}^\Delta(s))\bigr)\, \mathrm{d}W^0(s),
		\end{aligned}
	\end{equation}
	where
	\[
	\Delta_{t_s} := \Delta_{t_k} = \frac{h}{1 + \beta |b(X_{t_k}, \mathcal{L}^1(X_{t_k}))|} \le h, \quad {s} \in [t_k, t_{k+1}).
	\]

	\subsection{Reachability of \(\Delta_{t_k}\) over time \(T\)}
	\begin{thm}
		\label{thm:disc-reach}
		Suppose that assumptions \textup{\textbf{(A1)-(A7)}} hold. Let
		\(\Delta_{t_k}\) be given by \eqref{vvv}
		and define the stopping time
		\(\tau_T := \inf \left\{ n \ge 0 : S_n:=\sum_{k=0}^{n-1} \Delta_{t_k} \ge T \right\}.\)
		Then, $\mathbb{E}[\tau_T] < \infty$. In particular, $\tau_T$ is almost surely finite.
	\end{thm}
	In order to prove the aforementioned theorem, the following lemmas are presented.

\begin{lem}\label{thm:TAEM_moment_bounds}
	Suppose that assumptions \textbf{(A1)-(A7)} hold. Let $X^\Delta(t)$ be the continuous-time TAEM approximation given in \eqref{eq:TAEM-cont-expanded}. Then, for any real number $l \ge 2$, there exists a constant $C_{T,l}>0$, depending only on $T$, $l$ such that
	\[
	\sup_{t\le T} \mathbb E[|X^\Delta(t)|^l] \le C_{T,l}.
	\]
	In particular, since $X^\Delta(t_k) = X_{t_k}$ for all grid points $t_k$, we also have
    \[
    \sup_{t_k\le T} \mathbb{E}\big[|X_{t_k}|^l\big] \le C_{T,l}.
    \]
\end{lem}

	\begin{proof}
		We divide the proof into several steps.
		
		\paragraph{Step 1: It\^o's formula for $|X^\Delta(t)|^l$.}
		
		Let $\mu_s = \mathcal{L}^1(\overline{X}^\Delta(s))$ and define
		\begin{align*}
			f(s) &:= b_{\Delta_{t_s}}(\overline{X}^\Delta(s), \mu_s), \qquad
			g(s) := \sigma(\overline{X}^\Delta(s), \mu_s), \\
			g_0(s) &:= \sigma^0(\overline{X}^\Delta(s), \mu_s), \qquad
			j(s) := c(\overline{X}^\Delta(s), \mu_s).
		\end{align*}		
		Applying It\^o's formula to $|X^\Delta(t)|^l$ for $l \ge 2$
		\begin{equation}\label{mm}
			\begin{aligned}
				&|X^\Delta(t)|^l 
				= |X_0|^l + l \int_0^t |X^\Delta(s)|^{l-2} \,\langle X^\Delta(s), f(s)\rangle \,\mathrm ds \\
				&\quad + l \int_0^t |X^\Delta(s)|^{l-2} \,\langle X^\Delta(s), g(s)\,\mathrm dW(s)\rangle
				+ l \int_0^t |X^\Delta(s)|^{l-2} \,\langle X^\Delta(s), g_0(s)\,\mathrm dW^0(s)\rangle \\
				&\quad + \frac{l(l-2)}{2} \int_0^t |X^\Delta(s)|^{l-4} 
				\Big( |g(s)^\top X^\Delta(s)|^2 + |g_0(s)^\top X^\Delta(s)|^2 \Big)\,\mathrm ds \\
				&\quad + \frac{l}{2} \int_0^t |X^\Delta(s)|^{l-2} 
				\Big( \|g(s)\|^2 + \|g_0(s)\|^2 \Big)\,\mathrm ds \\
				&\quad + 2\int_0^t \int_{\mathbb{R}_0^d} 
				\Big[ |X^\Delta(s-) + j(s)z|^l - |X^\Delta(s-)|^l \Big]\,\tilde N(\mathrm ds,\mathrm dz) \\
				&\quad + \int_0^t \int_{\mathbb{R}_0^d} 
				\Big[ 2|X^\Delta(s-) + j(s)z|^l - 2|X^\Delta(s-)|^l 
				- l|X^\Delta(s-)|^{l-2}\,\langle X^\Delta(s-), j(s)z\rangle \Big]\,
				\nu(\mathrm dz)\,\mathrm ds .
			\end{aligned}
		\end{equation}

	\paragraph{Step 2: Drift term estimation.}

   For the drift term $f(s) := b_{\Delta_{t_s}}(\overline{X}^\Delta(s), \mu_s)$, by the taming property, we obtain
   \[
   |f(s)| \le |b(\overline{X}^\Delta(s), \mu_s)|.
   \] 
   Using \textbf{(A3)} with $\overline{x} = 0$ and $\overline{\mu} = \delta_0$, there exist constants \( \tilde{L} > 0 \) and \( \ell > 0 \) such that 
   \[
   |b(\overline{X}^\Delta(s), \mu_s)| \le |b(0, \delta_0)| + \tilde{L}(1 + |\overline{X}^\Delta(s)|^\ell)\left(|\overline{X}^\Delta(s)| + \mathcal{W}_2(\mu_s, \delta_0)\right).
   \]
   Since $\mathcal{W}_2(\mu_s, \delta_0) = \left(\mathbb{E}[|\overline{X}^\Delta(s)|^2\mid \mathcal{F}^0]\right)^{1/2}$, we have
   \begin{align*}
   	 |b(\overline{X}^\Delta(s), \mu_s)| &\le C + \tilde{L}(1 + |\overline{X}^\Delta(s)|^\ell)|\overline{X}^\Delta(s)| + \tilde{L}(1 + |\overline{X}^\Delta(s)|^\ell)\left(\mathbb{E}[|\overline{X}^\Delta(s)|^2\mid \mathcal{F}^0]\right)^{1/2}\\
   	 &\le C + \tilde{L}(1 + |\overline{X}^\Delta(s)|^\ell)|\overline{X}^\Delta(s)|+\frac{\tilde{L}}{2}(1 + |\overline{X}^\Delta(s)|^\ell)^2 + \frac{\tilde{L}}{2}\mathbb{E}[|\overline{X}^\Delta(s)|^2\mid \mathcal{F}^0]\\
   	 &\le C\left(1 + |\overline{X}^\Delta(s)|^{\ell+1} + |\overline{X}^\Delta(s)|^{2\ell} + \mathbb{E}[|\overline{X}^\Delta(s)|^2\mid \mathcal{F}^0]\right).
   \end{align*}
 Now estimate the drift term, by choosing $l\geq \max\{\,2\ell,\ell+1,2\}, 
$ we have
 \begin{align*}
	&\left| l \int_0^t |X^\Delta(s)|^{l-2} \langle X^\Delta(s), f(s) \rangle \,\mathrm{d}s \right| \le l \int_0^t |X^\Delta(s)|^{l-1} |f(s)| \,\mathrm{d}s \\
	&\le C_l \int_0^t |X^\Delta(s)|^{l-1} \left(1 + |\overline{X}^\Delta(s)|^{\ell+1} + |\overline{X}^\Delta(s)|^{2\ell} + \mathbb{E}[|\overline{X}^\Delta(s)|^2\mid \mathcal{F}^0]\right) \,\mathrm{d}s\\
	& \le C_l \int_0^t \left(1 + |X^\Delta(s)|^l + |\overline{X}^\Delta(s)|^l + \mathbb{E}[|\overline{X}^\Delta(s)|^l\mid \mathcal{F}^0]\right) \,\mathrm{d}s.
 \end{align*}

  \paragraph{Step 3: Diffusion terms estimation.}
		
  From Assumption \textbf{(A1)}, we have
\[	|g(s)|^2 + |g_0(s)|^2 \le L\left(1 + |\overline{X}^\Delta(s)|^2 + \mathcal{W}_2^2(\mu_s, \delta_0)\right)\le L\left(1 + |\overline{X}^\Delta(s)|^2 + \mathbb{E}[|\overline{X}^\Delta(s)|^2\mid \mathcal{F}^0]\right).
\]
Then we obtain
\begin{align*}
	&\frac{l(l-2)}{2} \int_0^t |X^\Delta(s)|^{l-4} 
	\Big( |g(s)^\top X^\Delta(s)|^2 + |g_0(s)^\top X^\Delta(s)|^2 \Big)\,\mathrm ds + \frac{l}{2} \int_0^t |X^\Delta(s)|^{l-2} 
	\Big( \|g(s)\|^2 + \|g_0(s)\|^2 \Big)\,\mathrm ds \\& \leq C_l \int_0^t |X^\Delta(s)|^{l-2} \left( |g(s)|^2 + |g_0(s)|^2 \right) \,\mathrm{d}s\\
	&\leq C_{l,L} \int_0^t \left[ 1 + |X^\Delta(s)|^l + |\overline{X}^\Delta(s)|^l + \mathbb{E}[|\overline{X}^\Delta(s)|^l\mid \mathcal{F}^0] \right] \,\mathrm{d}s.
\end{align*}
	
		\paragraph{Step 4: Jump term estimation.}
		
		Write the jump contribution as $M_J(t)+D_J(t)$, where the martingale part is
		\[
			M_J(t)
			= 2\int_0^t\!\!\int_{\mathbb{R}_0^d} 
			\Big[|X^\Delta(s{-})+j(s)z|^{l}-|X^\Delta(s{-})|^{l}
			\Big]\,
			\tilde N(\mathrm dz,\mathrm ds),
		\]
		and the compensator is
		\[
			D_J(t)
			= \int_0^t \int_{\mathbb{R}_0^d} 
			\Big[ 2|X^\Delta(s-) + j(s)z|^l - 2|X^\Delta(s-)|^l 
			- l|X^\Delta(s-)|^{l-2}\,\langle X^\Delta(s-), j(s)z\rangle \Big]\,
			\nu(\mathrm dz)\,\mathrm ds .
		\]
		For $l\ge 2$, by Taylor’s formula there exists $C_l>0$ such that for all $x,z$,
		\[
		\big||x+jz|^{l}-|x|^{l}-l|x|^{l-2}\langle x,jz\rangle\big|
		\;\le\; C_l\big(|x|^{l-2}|jz|^{2}+|jz|^{l}\big),
		\]
		and also
		\[
		\big||x+jz|^{l}-|x|^{l}\big|
		\;\le\; C_l\big(|x|^{l-1}|jz|^{2}+|jz|^{l}\big).
		\]
		Hence,
		\[
			|D_J(t)|
			\le C_l \int_0^t\!\!\int_{\mathbb{R}_0^d}
			\Big(|X^\Delta(s)|^{l-2}|j(s)|^{2}|z|^{2}+2|j(s)|^{l}|z|^{l}+|X^\Delta(s)|^{l-1}|j(s)|^{2}|z|^{2}\Big)\,\nu(\mathrm dz)\,\mathrm ds.
		\]	
		By Assumption \textbf{(A5)}, we have
		\[
		|j(s)|
		\;\le\;L_3\!\left(1+|\overline X^\Delta(s)|
		+\big[\mathbb E\!\left(|\overline X^\Delta(s)|^{2}\mid\mathcal F^{0}\right)\big]^{1/2}\right),
		\]
		so that $|j(s)|^{2}$ and $|j(s)|^{l}$ are bounded by
		\(
		1+|\overline X^\Delta(s)|^{l}+\mathbb E\!\left(|\overline X^\Delta(s)|^{l}\mid\mathcal F^{0}\right).\)
		Moreover, Assumption \textbf{(A4)} ensures 
	        \(
			\int_{|z|<1}|z|^{2}\nu(\mathrm dz)<\infty
			\)
			and
			\(
			\int_{|z|\geq 1}|z|^{l}\nu(\mathrm dz)<\infty.
			\)
		Collecting the above, we obtain
		\[
		|D_J(t)|
		\;\le\; C_{l,L_3,\nu}\!\int_0^t
		\Big[1+|X^\Delta(s)|^{l}
		+|\overline X^\Delta(s)|^{l}
		+\mathbb E\!\left(|\overline X^\Delta(s)|^{l}\mid\mathcal F^{0}\right)\Big]\,
		\mathrm ds .
		\]

		\paragraph{Step 5: Taking expectations and Gr\"onwall's inequality.}
		
		Taking expectations to \eqref{mm}, all martingale terms vanish. Let
		\[
		m(t) = \sup_{0 \le u \le t} \mathbb{E}[|X^\Delta(u)|^l],
		\]
		we obtain
		\[
		\mathbb{E}[|X^\Delta(t)|^l] \le \mathbb{E}[|X_0|^l] + C \int_0^t \left[ 1 + \mathbb{E}[|X^\Delta(s)|^l] + \mathbb{E}[|\overline{X}^\Delta(s)|^l] \right] \,\mathrm{d}s.
		\]
		Since $\overline{X}^\Delta(s) = X_{t_k}$ for $s \in [t_k, t_{k+1})$, we have $\mathbb{E}[|\overline{X}^\Delta(s)|^l] \le m(s)$ for $s \le t$. Therefore,
		\[
		m(t) \le \mathbb{E}[|X_0|^l] + C \int_0^t \left[ 1 + 2m(s) \right] \,\mathrm{d}s.
		\]
		By Gr\"onwall's inequality, we get
		\[
		m(T) \le \left( \mathbb{E}[|X_0|^l] + CT \right) e^{2CT} =: C_{T,l}.
		\]
		This completes the proof of the lemma.
	\end{proof}

	\begin{proof}[\textbf{Proof of Theorem \ref{thm:disc-reach}}]
		According to Assumption \textbf{(A3)}, we have
		\[
		|b(x,\mu)-b(0,\delta_0)| \le \tilde{L}(1+|x|^\ell+|0|^\ell)\,(|x-0|+\mathcal W_2(\mu,\delta_0)) 
		= \tilde{L}(1+|x|^\ell)(|x|+\mathcal W_2(\mu,\delta_0)).
		\]
		Taking
		$x=X_{t_k},\ \mu=\mathcal L^1(X_{t_k})$, we have
		\begin{align*}
			|b(X_{t_k},\mathcal L^1(X_{t_k})|
			&\le C_0 + L(1+|X_{t_k}|^\ell)(|X_{t_k}|+\mathcal W_2(\mathcal L^1(X_{t_k},\delta_0))\\
			&\le C_0 + L(1+|X_{t_k}|^\ell)\Big(|X_{t_k}|+(\mathbb E[|X_{t_k}|^2\mid \mathcal{F}^0])^{1/2}\Big),
		\end{align*}
		where \(C_0=|b(0,\delta_0)|. \)
		By H\"older’s inequality and Lemma~\ref{thm:TAEM_moment_bounds}, both
		$\mathbb E[(1+|X_{t_k}|^\ell)|X_{t_k}|]$ and 
		$\mathbb E[1+|X_{t_k}|^\ell\cdot(\mathbb E|X_{t_k}|^2)^{1/2}]$
		are bounded by a constant depending only on $L$ and $C_{T,\ell}$.
		Hence there exists $M = M(L,C_{T,\ell})$ such that,
		\[
		\mathbb E[|b(X_{t_k},\mathcal L^1(X_{t_k}))|]
		\le C_0 + L\mathbb E\big[(1+|X_{t_k}|^\ell)|X_{t_k}|\big]
		+ L\mathbb E[1+|X_{t_k}|^\ell]\cdot (\mathbb E[|X_{t_k}|^2])^{1/2}
		\le M.
		\]	
		Define \(f(y)=\dfrac{h}{1+\beta y}\) for \(y\ge0\). Note \(f\) is convex and nonnegative on \([0,\infty)\). By Jensen's inequality,
		\[
		\mathbb{E}[\Delta_{t_k}]=\mathbb{E}\big[f(|b(X_{t_k},\mathcal L^1(X_{t_k}))|)\big]
		\ge f\big(\mathbb{E}[|b(X_{t_k},\mathcal L^1(X_{t_k}))|]\big) \ge \frac{h}{1+\beta M}=:\,c>0.
		\]
		Due to \(S_{\tau_T}-T\leq \Delta_{t_{\tau_T-1}}\leq h\). Fix a constant truncation integer $N$, we have
		\begin{equation}\label{nn}
			S_{\tau_T\wedge N} \leq S_{\tau_T}\le T + h.
		\end{equation}
		Taking expectations to \eqref{nn}, we have
		\begin{align*}
		\mathbb{E}[S_{\tau_T\wedge N}]
		&=\mathbb{E}[\sum_{k=0}^{\tau_T\wedge N-1}\Delta_{t_k}]=\mathbb{E}[\sum_{k=0}^{N-1}\Delta_{t_k}\mathbf{1}_{\{\tau_T\wedge N>k\}}]= \sum_{k=0}^{N-1}\mathbb{E}[\Delta_{t_k} \mathbf{1}_{\{\tau_T>k\}}]\\
		&\ge \sum_{k=0}^{N-1} \mathbb{E}[\Delta_{t_k}]\,\mathbb{E}[ \mathbf{1}_{\{\tau_T>k\}}] 
		\ge c\,\mathbb{E}[\sum_{k=0}^{N-1} \mathbf{1}_{\{\tau_T>k\}}]
		= c\,\mathbb{E}[\tau_T\wedge N].
		\end{align*}
		Combining with \(\mathbb{E}[S_{\tau_T\wedge N}]\le T+h\) yields
		\[
		c\,\mathbb{E}[\tau_T\wedge N]\le T+h.
		\]
		Let \(N\to\infty\) and apply monotone convergence to get \(c\,\mathbb{E}[\tau_T]\le T+h\), so we have \(\mathbb{E}[\tau_T]<\infty\). Therefore \(\mathbb{P}(\tau_T=\infty)=0\) and \(\tau_T<\infty\) almost surely.
	\end{proof}
	\subsection{Convergence of TAEM Method under assumptions (A1)-(A7)}
	
	We now establish strong convergence of the TAEM scheme \eqref{eq:TAEM-discrete}-\eqref{eq:TAEM-cont-expanded} under assumptions \textbf{(A1)}-\textbf{(A7)}. Fix $R>0$ and define the stopping time
	\begin{equation}\label{rrr}
		\tau_R := \inf \bigl\{ t \ge 0 : |X(t)| \ge R \ \text{or}\ |X^\Delta(t)| \ge R \bigr\}.
	\end{equation}
	Consider the stopped processes
	\[
	X^R(t) := X(t \wedge \tau_R), \qquad X^{\Delta,R}(t) := X^\Delta(t \wedge \tau_R),
	\]
	and set
	\[
	e^R(t) := X^{\Delta,R}(t) - X^R(t), \qquad  e(t):=X^\Delta(t) - X(t).
	\]
\begin{thm}\label{cc}
	Suppose that assumptions \textup{\textbf{(A1)-(A7)}} hold. 
	Let \(X(t)\) be the unique strong solution of \eqref{eq:MKV-SDE}, and 
	let \(X^\Delta(t)\) be the continuous-time TAEM approximation given in \eqref{eq:TAEM-cont-expanded}, with
	adaptive time step
	\[
	\Delta_{t_k} = \frac{h}{1+\beta\,|b(X_{t_k},\mathcal L^1(X_{t_k}))|}.
	\]
	Then for any \(\varepsilon\in(0,1)\) and any \(p>2\), there exist constants \(C_{T,\varepsilon,p}>0\) and \(C_{T,p}>0\), independent of \(h\), and a constant \(C>0\) depending only on \(\ell>0\), such that
	\[
	\mathbb{E}\!\left[ \sup_{0 \le t \le T} |X^\Delta(t) - X(t)|^2 \right] 
	\;\le\;
	C_{T,\varepsilon,p}\, h^{1 - \varepsilon} \, (\log(1/h))^{C}
	\;+\;
	C_{T,p}\,(\log(1/h))^{-\frac{p-2}{2\ell}} .
	\]
	In particular, the TAEM method is strongly convergent in the sense that
	\[
	\lim_{h\to 0}\mathbb{E}\!\left[ \sup_{0 \le t \le T} |X^\Delta(t) - X(t)|^2 \right] = 0 .
	\]
\end{thm}

	In order to prove the aforementioned theorem, the following lemmas are presented.

		\begin{lem}\label{ll}
		Let $\overline{X}^\Delta(t)$ be the piecewise constant interpolant defined in \eqref{eq:step-interpolant}, $X^\Delta(t)$ be the continuous-time TAEM approximation defined in \eqref{eq:TAEM-cont-expanded}, and $\tau_R$ be the stopping time defined in \eqref{rrr},
		there exists a constant $C_{R}>0$ such that
		\[
			\int_0^T \mathbb E|X^\Delta(s)-\overline X^\Delta(s)|^2\,{\rm{d}}s\ \le\ T\,C_R\,h.
		\]
	\end{lem}

	\begin{proof}
		$\exists~ t_k < \tau_R$, for $s \in [t_k, t_{k+1}) \cap [0,\tau_R]$, by \eqref{eq:TAEM-cont-expanded} we have
		\begin{align*}
			X^\Delta(s) &= X_{t_k}
			+ \int_{t_k}^s b_{\Delta_{t_k}}\bigl(X_{t_k},\mathcal{L}^1(X_{t_k})\bigr)\, \mathrm{d}r
			+ \int_{t_k}^s \sigma\bigl(X_{t_k},\mathcal{L}^1(X_{t_k})\bigr)\, \mathrm{d}W(r) \\
			&\quad + \int_{t_k}^s c\bigl(X_{t_k},\mathcal{L}^1(X_{t_k})\bigr)\, \mathrm{d}Z(r) 
			+ \int_{t_k}^s \sigma^0\bigl(X_{t_k},\mathcal{L}^1(X_{t_k})\bigr)\, \mathrm{d}W^0(r).
		\end{align*}
		Hence,
		\begin{align*}
			\mathbb{E}\big|X^\Delta(s) - \overline{X}^\Delta(s)\big|^2
			&\leq 4\mathbb{E}\bigg[\Big|\int_{t_k}^s b_{\Delta_{t_k}}(X_{t_k},\mathcal{L}^1(X_{t_k}))\, {\rm{d}}r \Big|^2+ \Big|\int_{t_k}^s \sigma(X_{t_k},\mathcal{L}^1(X_{t_k}))\, {\rm{d}}W(r) \Big|^2\\
			&\quad + \Big|\int_{t_k}^s c(X_{t_k},\mathcal{L}^1(X_{t_k}))\, {\rm{d}}Z(r) \Big|^2+\Big|\int_{t_k}^s \sigma^0(X_{t_k},\mathcal{L}^1(X_{t_k}))\, {\rm{d}}W^0(r) \Big|^2\bigg].
		\end{align*}
		\textbf{Step 1. Bound of the drift term.}
		By Assumption \textbf{(A3)} and the definition of $\tau_R$, we have $|X_{t_k}|\le R$ and $\mathcal W_2(\mathcal L^1(X_{t_k}),\delta_0)\le R $. Hence,
		\[|b_{\Delta_{t_k}}(X_{t_k},\mathcal{L}^1(X_{t_k}))|\leq
		|b(X_{t_k},\mathcal L^1(X_{t_k}))|
		\le |b(0,\delta_0)| + 2\tilde L(1+R^\ell)R =: B_R.
		\]
		Then
		\[
		\mathbb E\Big|\int_{t_k}^s b_{\Delta_{t_k}}(X_{t_k},\mathcal{L}^1(X_{t_k}))\, \mathrm{d}r\Big|^2
		\le (s-t_k)^2 B_R^2.
		\]
		\textbf{Step 2. Bound of the diffusion and jump terms.}
		By Assumption \textbf{(A7)}, for any $|x|$, $\mathcal W_2(\mathcal L^1(X_{t_k}),\delta_0)\le R $, we have
		\[
		|\sigma(x,\mu)|^2 + |\sigma^{0}(x,\mu)|^2 + |c(x,\mu)|^2\!\int_{\mathbb R_0^d} |z|^2\nu(\mathrm{d}z)
		\le 2\bigl(|\sigma(0,\delta_0)|^2 + |\sigma^{0}(0,\delta_0)|^2 + |c(0,\delta_0)|^2I_2\bigr)
		+ 2(\tilde L_1+\tilde L_2)R^2,
		\]
		where \(I_2 := \int_{\mathbb R_0^d} |z|^2\nu(\mathrm{d}z)\).
		Denote this constant by
		\[
		K_R := 2\bigl(|\sigma(0,\delta_0)|^2 + |\sigma^{0}(0,\delta_0)|^2 + |c(0,\delta_0)|^2I_2\bigr)
		+ 2(\tilde L_1+\tilde L_2)R^2.
		\]
		Then, by the It\^o isometry and the corresponding property for L\'evy jumps,
		\begin{align*}
			&\mathbb{E}\Big|\int_{t_k}^s \sigma(X_{t_k},\mathcal{L}^1(X_{t_k}))\,\mathrm{d}W(r)\Big|^2
			+\mathbb{E}\Big|\int_{t_k}^s \sigma^{0}(X_{t_k},\mathcal{L}^1(X_{t_k}))\,\mathrm{d}W^0(r)\Big|^2\\
			&\qquad
			+\mathbb{E}\Big|\int_{t_k}^s c(X_{t_k},\mathcal{L}^1(X_{t_k}))\,\mathrm{d}Z(r)\Big|^2\le 3K_R (s-t_k).
		\end{align*}
		
		Combining the above bounds yields
		\[
		\mathbb E|X^\Delta(s)-\overline X^\Delta(s)|^2
		\le 4\Big[(s-t_k)^2 B_R^2 + 3K_R (s-t_k)\Big].
		\]
	Therefore, since $s\in[0,T]$ is arbitrary, we obtain
		\[
		\mathbb E|X^\Delta(s)-\overline X^\Delta(s)|^2 \le 4\big(B_R^2 + 3K_R\big)h =: C_R\,h,\qquad \forall s\in[0,T].
		\]
		Hence,
		\[
	\int_0^T \mathbb E|X^\Delta(s)-\overline X^\Delta(s)|^2\,{\rm{d}}s\ \le\ T\,C_R\,h,
		\]
		where
		\begin{equation}\label{oo}
			\begin{aligned}
				&C_{R} := 4\big(B_R^2 + 3K_R\big),
				\quad
				B_R = |b(0,\delta_0)| + 2\tilde L(1+R^\ell)R,\\&
				K_R = 2\bigl(|\sigma(0,\delta_0)|^2 + |\sigma^{0}(0,\delta_0)|^2 + |c(0,\delta_0)|^2I_2\bigr)
				+ 2(\tilde L_1+\tilde L_2)R^2.
			\end{aligned}
		\end{equation}
	\end{proof}
	
\begin{lem}\label{final} 
	Suppose that there exist constants \( C_{R,1}, C_R^{(\ast)} > 0 \) such that
	\[
	\mathbb{E}\bigg[ \sup_{0 \le t \le T} |e^R(t)|^2 \bigg]
	\le 5 C_{R,1} h^{2\alpha}
	+ C_R^{(\ast)} \int_0^T \mathbb{E}\big[ |\overline{X}^\Delta(s) - X(s)|^2 \cdot \mathbf{1}_{\{s \le \tau_R\}} \big] \,{\rm d}s.
	\]
	Then,
	\[
	\mathbb{E}\bigg[ \sup_{0 \le t \le T} |e^R(t)|^2 \bigg]
	\le \Big( 5 C_{R,1} h^{2\alpha} + 8 C_R^{(\ast)} T (B_R^2 + 3 K_R) h \Big) e^{2 C_R^{(\ast)} T}.
	\]
\end{lem}

\begin{proof} 
	Due to Lemma~\ref{ll}, we have
	\begin{equation}\label{ccc}
		\begin{aligned}      &\mathbb{E}\Big[\sup_{0 \le t \le T} |e^R(t)|^2 \Big] \le 5 C_{R,1} h^{2\alpha}
			+ C_R^{(\ast)} \int_0^T \mathbb E\big|\overline X^\Delta(s)-X(s)\big|^2\,\mathbf{1}_{\{s\le \tau_R\}}\,{\rm d}s\\
			&\le 5 C_{R,1} h^{2\alpha}
			+ C_R^{(\ast)} \int_0^T\!\Big(
			2\,\mathbb{E}\Big[\sup_{0\le r\le s}|e^R(r)|^2\Big]
			+ 2\,\mathbb{E}\big|X^\Delta(s)-\overline X^\Delta(s)\big|^2
			\Big){\rm d}s \\
			&\le 5 C_{R,1} h^{2\alpha}
			+ 2C_R^{(\ast)} \int_0^T \mathbb{E}\Big[\sup_{0\le r\le s}|e^R(r)|^2\Big]\,{\rm d}s
			+ 8C_R^{(\ast)}T\big(B_R^2+3K_R\big)\,h.
		\end{aligned} 
	\end{equation} 
 By using Gr\"onwall's inequality to \eqref{ccc}, we have \[
\mathbb E\Big[\sup_{0\le t\le T}|e^R(t)|^2\Big]
\le \big(5 C_{R,1} h^{2\alpha}+ 8C_R^{(\ast)}T(B_R^2+3K_R)\,h\big)\,e^{2C_R^{(\ast)}T}.
\]
\end{proof}

	\begin{proof}[\textbf{Proof of Theorem \ref{cc}}]
		For $t\in [0,\tau_R]$, the coefficients are globally Lipschitz with constants depending only on $R$.		
		Note that $e^R(t) = X^{\Delta,R}(t) - X^R(t)$. Subtracting the integral equations gives
		\[
		\begin{aligned}
			e^R(t) &= \int_0^{t \wedge \tau_R} \big( b_{\Delta_{t_s}}(\overline{X}^\Delta(s),\bar{\mu}_s) - b(\overline{X}^\Delta(s),\bar{\mu}_s) \big) \, \mathrm{d}s  + \int_0^{t \wedge \tau_R} \big( b(\overline{X}^\Delta(s),\bar{\mu}_s) - b(X(s),\mu_s) \big) \, \mathrm{d}s \\
			&\quad + \int_0^{t \wedge \tau_R} \big( \sigma(\overline{X}^\Delta(s),\bar{\mu}_s) - \sigma(X(s),\mu_s) \big) \, \mathrm{d}W(s)  + \int_0^{t \wedge \tau_R} \big( c(\overline{X}^\Delta(s),\bar{\mu}_s) - c(X(s),\mu_s) \big) \, \mathrm{d}Z(s) \\
			&\quad + \int_0^{t \wedge \tau_R} \big( \sigma^0(\overline{X}^\Delta(s),\bar{\mu}_s) - \sigma^0(X(s),\mu_s) \big) \, \mathrm{d}W^0(s) =:e_1(t)+e_2(t)+e_3(t)+e_4(t)+e_5(t),
		\end{aligned}
		\]
		where $\bar{\mu}_s := \mathcal{L}^1(\overline{X}^\Delta(s))$, $\mu_s := \mathcal{L}^1(X(s))$. 
		Applying $(a_1+\cdots+a_5)^2 \le 5(a_1^2+\cdots+a_5^2)$ and the BDG inequality, we obtain
		\begin{equation}\label{eR}
			\begin{aligned}
				\mathbb{E}\Big[\sup_{0 \le t \le T} |e^R(t)|^2 \Big]
				&\le 5\,\mathbb{E}\Big[\sup_{0 \le t \le T} |e_1(t)|^2 \Big] + 5\,\mathbb{E}\Big[\sup_{0 \le t \le T} |e_2(t)|^2 \Big] 
				+ 5\,\mathbb{E}\Big[\sup_{0 \le t \le T} |e_3(t)|^2 \Big] \\
				&\quad + 5\,\mathbb{E}\Big[\sup_{0 \le t \le T} |e_4(t)|^2 \Big] + 5\,\mathbb{E}\Big[\sup_{0 \le t \le T} |e_5(t)|^2 \Big] \\
				&=: 5(I_1 + I_2 + I_3 + I_4 + I_5),
			\end{aligned}
		\end{equation}
		where \(I_i=\mathbb{E}\big[\sup_{0\le t\le T}|e_i(t)|^2\big]\).

		\noindent\textbf{Step 1: Estimate of $I_1$.}
		On $[0,\tau_R]$ we have $|\overline X^\Delta(s)|\le R$ and $\mathcal W_2(\bar\mu_s,\delta_0)\le R$. By Assumption \textbf{(A3)},
		\[
		|b(\overline X^\Delta(s),\bar\mu_s)|
		\le |b(0,\delta_0)| + 2\tilde L (1+R^\ell)R
		= B_R .
		\]
		Hence for $s\le\tau_R$,
		\[
		\big|b_{\Delta_{t_s}}(\overline X^\Delta(s),\bar\mu_s)-b(\overline X^\Delta(s),\bar\mu_s)\big|
		\le \Delta_{t_s}^\alpha B_R^2 \le h^\alpha B_R^2.
		\]
		Therefore,
		\[
		\sup_{0\le t\le T}\Big|\int_0^{t\wedge\tau_R}\big(b_{\Delta_{t_s}}(\overline X^\Delta(s),\bar\mu_s)-b(\overline X^\Delta(s),\bar\mu_s)\big){\rm d}s\Big|^2
		\le (T h^\alpha B_R^2)^2
		= T^2 B_R^4\, h^{2\alpha}.
		\]
		Taking expectation gives
		\begin{equation}\label{11}
		I_1 \le C_{R,1} h^{2\alpha}
		\quad\text{with}\quad
		C_{R,1} := T^2 B_R^4=T^2 \big[|b(0,\delta_0)| + 2\tilde L(1+R^\ell)R\big]^4.
		\end{equation}
		\noindent\textbf{Step 2: Estimate of $I_2$.}
		By Assumption \textbf{(A3)}, for $s\le \tau_R$,
		\[
		|b(\overline X^\Delta(s),\bar\mu_s)-b(X(s),\mu_s)|
		\le L_b(R)\big(|\overline X^\Delta(s)-X(s)| + \mathcal W_2(\bar\mu_s,\mu_s)\big),
		\]
		where
		\[
		L_b(R) := \tilde L(1+2R^\ell).
		\]
		Then, using $\sup_{0\le t\le T}|\int_0^t f(s){\rm d}s|^2 \le T \int_0^T |f(s)|^2{\rm d}s$ and
		\(\mathcal W_2^2(\bar\mu_s,\mu_s)\le \mathbb E|\overline X^\Delta(s)-X(s)|^2\), we have
		\[
		\begin{aligned}
			I_2
			&= \mathbb E\Big[\sup_{0\le t\le T}\Big|\int_0^{t\wedge\tau_R}\big(b(\overline X^\Delta,\bar\mu)-b(X,\mu)\big){\rm d}s\Big|^2\Big] \\
			&\le T \mathbb E\int_0^{T\wedge\tau_R} |b(\overline X^\Delta(s),\bar\mu_s)-b(X(s),\mu_s)|^2 {\rm d}s \\
			&\le T \cdot 2 L_b(R)^2 \int_0^{T\wedge\tau_R} \mathbb E\Big[\,|\overline X^\Delta(s)-X(s)|^2 + \mathcal W_2^2(\bar\mu_s,\mu_s)\Big]{\rm d}s \\
			&\le T \cdot 2 L_b(R)^2 \int_0^{T\wedge\tau_R} \mathbb E\Big[\,|\overline X^\Delta(s)-X(s)|^2 + \mathbb E|\overline X^\Delta(s)-X(s)|^2\Big]{\rm d}s \\
			&\le 4 T L_b(R)^2 \int_0^{T\wedge\tau_R} \mathbb E|\overline X^\Delta(s)-X(s)|^2{\rm d}s.
		\end{aligned}
		\]
		Hence we can take
		\begin{equation}\label{22}
			C_{R,2} := 4 T \bigl(\tilde L(1+2R^\ell)\bigr)^2
		\quad\text{and}\quad
		I_2 \le C_{R,2} \int_0^{T\wedge\tau_R} \mathbb E|\overline X^\Delta(s)-X(s)|^2{\rm d}s.
		\end{equation}
		\noindent\textbf{Step 3: Estimate of $I_3$.}
		Let $g(s):=\sigma(\overline X^\Delta(s),\bar\mu_s)-\sigma(X(s),\mu_s)$. By the BDG inequality,
		\[
		\mathbb E\Big[\sup_{0\le t\le T}\Big|\int_0^{t\wedge\tau_R} g(s)\,{\rm d}W(s)\Big|^2\Big]
		\le C_{\mathrm{BDG}} \,\mathbb E\int_0^{T\wedge\tau_R} |g(s)|^2 {\rm d}s.
		\]
		By Assumption \textbf{(A7)},
		\[
		|g(s)|^2 \le \tilde L_1 |\overline X^\Delta(s)-X(s)|^2 + \tilde L_2 \mathcal W_2^2(\bar\mu_s,\mu_s)
		\le (\tilde L_1 + \tilde L_2)\, \mathbb E|\overline X^\Delta(s)-X(s)|^2.
		\]
		Hence,
		\[
		I_3 \le C_{\mathrm{BDG}}(\tilde L_1 + \tilde L_2) \int_0^{T\wedge\tau_R} \mathbb E|\overline X^\Delta(s)-X(s)|^2{\rm d}s.
		\]
		So we set
		\begin{equation}\label{33}
			C_{R,3} := C_{\mathrm{BDG}}(\tilde L_1 + \tilde L_2)
			\quad\text{and}\quad
			I_3 \le C_{R,3} \int_0^{T\wedge\tau_R} \mathbb E|\overline X^\Delta(s)-X(s)|^2{\rm d}s.
		\end{equation}
		\noindent\textbf{Step 4: Estimate of $I_4$.}
		For the jump part, let
		\[
		h(s,z) := \big(c(\overline X^\Delta(s),\bar\mu_s)-c(X(s),\mu_s)\big) z,
		\]
		and let $M_t := \int_0^t \int_{\mathbb R_0^d} h(s,z)\,\tilde N({\rm d}s,{\rm d}z)$. By Kunita’s inequality,
		\[
		I_4 = \mathbb E\Big[\sup_{0\le t\le T} |M_{t\wedge\tau_R}|^2\Big]
		\le C_{\mathrm{Ku}} \,\mathbb E\int_0^{T\wedge\tau_R} \int_{\mathbb R_0^d} |h(s,z)|^2 \nu({\rm d}z)\,{\rm d}s.
		\]
		Using assumptions \textbf{(A4)} and \textbf{(A7)},
		\begin{align*}
			\int_{\mathbb R_0^d} |h(s,z)|^2 \nu({\rm d}z)
			&= |c(\overline X^\Delta(s),\bar\mu_s)-c(X(s),\mu_s)|^2 \int_{\mathbb R_0^d} |z|^2 \nu({\rm d}z)\\
			&\le (\tilde L_1 + \tilde L_2)\big(|\overline X^\Delta(s)-X(s)|^2 + \mathcal W_2^2(\bar\mu_s,\mu_s)\big).
		\end{align*}
		Due to $\mathcal W_2^2(\bar\mu_s,\mu_s)\le \mathbb E|\overline X^\Delta(s)-X(s)|^2$, we have
		\[
		I_4 \le 2 C_{\mathrm{Ku}} (\tilde L_1 + \tilde L_2) \int_0^{T\wedge\tau_R} \mathbb E|\overline X^\Delta(s)-X(s)|^2{\rm d}s.
		\]
		Hence,
		\begin{equation}\label{44}
			C_{R,4} := 2 C_{\mathrm{Ku}} (\tilde L_1 + \tilde L_2)
			\quad\text{and}\quad
			I_4 \le C_{R,4} \int_0^{T\wedge\tau_R} \mathbb E|\overline X^\Delta(s)-X(s)|^2{\rm d}s.
		\end{equation}
		\noindent\textbf{Step 5: Estimate of $I_5$.}
		This is completely analogous to the estimate of \(I_3\). Thus, 
		\begin{equation}\label{55}
			I_5 \le C_{\mathrm{BDG}}(\tilde L_1 + \tilde L_2) \int_0^{T\wedge\tau_R} \mathbb E|\overline X^\Delta(s)-X(s)|^2{\rm d}s=:C_{R,5}\int_0^{T\wedge\tau_R} \mathbb E|\overline X^\Delta(s)-X(s)|^2{\rm d}s .
		\end{equation}
		Putting \eqref{11}, \eqref{22}, \eqref{33}, \eqref{44}, \eqref{55} back into \eqref{eR}, we obtain
		\begin{equation}\label{ww}
			\mathbb E\Big[\sup_{0\le t\le T} |e^R(t)|^2\Big]
			\le 5 C_{R,1} h^{2\alpha}
			+ 5\big( C_{R,2} + C_{R,3} + C_{R,4} + C_{R,5} \big)
			\int_0^{T\wedge\tau_R} \mathbb E|\overline X^\Delta(s)-X(s)|^2{\rm d}s.
		\end{equation}
		If we denote
		\begin{equation}\label{mmm}
			C_R^{(\ast)} := 5\big( C_{R,2} + C_{R,3} + C_{R,4} + C_{R,5} \big),
		\end{equation}
		then the estimate \eqref{ww} becomes
		\[\mathbb E\Big[\sup_{0\le t\le T} |e^R(t)|^2\Big]
		\le 5 C_{R,1} h^{2\alpha}
		+ C_R^{(\ast)}
		\int_0^{T\wedge\tau_R} \mathbb E|\overline X^\Delta(s)-X(s)|^2{\rm d}s.\]
To obtain the global error estimate for $e(t) = X^\Delta(t) - X(t)$, 
\begin{align*}
	\mathbb{E}\left[ \sup_{0 \le t \le T} |e(t)|^2 \right] 
	&\leq \mathbb{E}\left[ \sup_{0 \le t \le T} |e^R(t)|^2 \right] + \mathbb{E}\left[ \sup_{0 \le t \le T} |e(t)|^2 \cdot \mathbf{1}_{\{\tau_R \le T\}} \right].
\end{align*}
According to Proposition~\ref{prp:2.3} and Lemma~\ref{thm:TAEM_moment_bounds}, for $p \geq 2$ we have,
\begin{equation}\label{fff}
	\mathbb{E}\left[ \sup_{0 \le t \le T} |X(t)|^p \right] \le C_{T,p}, \quad \mathbb{E}\left[ \sup_{0 \le t \le T} |X^\Delta(t)|^p \right] \le C_{T,p}.
\end{equation}
Now, by using H\"older's inequality,
\begin{align*}
	\mathbb{E}\left[ \sup_{0 \le t \le T} |e(t)|^2 \cdot \mathbf{1}_{\{\tau_R \le T\}} \right] 
	&\le \left( \mathbb{E}\left[ \sup_{0 \le t \le T} |e(t)|^p \right] \right)^{2/p} \cdot \left( \mathbb{P}(\tau_R \le T) \right)^{1 - 2/p}.
\end{align*}
By the triangle inequality and \eqref{fff}, we obtain
\[
\mathbb{E}\left[ \sup_{0 \le t \le T} |e(t)|^p \right] \le 2^{p-1} \left( \mathbb{E}\left[ \sup_{0 \le t \le T} |X(t)|^p \right] + \mathbb{E}\left[ \sup_{0 \le t \le T} |X^\Delta(t)|^p \right] \right) \le 2^p C_{T,p}.
\]	
Now, by Markov's inequality,
\[
\mathbb{P}(\tau_R \le T) \le \mathbb{P}\left( \sup_{0 \le t \le T} |X(t)| \ge R \right) + \mathbb{P}\left( \sup_{0 \le t \le T} |X^\Delta(t)| \ge R \right) \le \frac{2C_{T,p}}{R^p}.
\]
Thus,
\begin{equation}\label{77}
	\mathbb{E}\left[ \sup_{0 \le t \le T} |e(t)|^2 \cdot \mathbf{1}_{\{\tau_R \le T\}} \right] \le (2^p C_{T,p})^{2/p} \cdot \left( \frac{2C_{T,p}}{R^p} \right)^{1 - 2/p} \le C_{T,p} R^{2-p},
\end{equation}
combining with Lemma~\ref{final} we have
\begin{equation}\label{x}
	\begin{aligned}
		\mathbb{E}\left[ \sup_{0 \le t \le T} |e(t)|^2 \right]
		&\le 5 C_{R,1} h^{2\alpha}
		+ C_R^{(\ast)}
		\int_0^{T\wedge\tau_R} \mathbb E|\overline X^\Delta(s)-X(s)|^2{\rm d}s + \mathbb{E}\left[ \sup_{0 \le t \le T} |e(t)|^2 \cdot \mathbf{1}_{\{\tau_R \le T\}} \right]\\
		&\le \Big( 5 C_{R,1} h^{2\alpha}
		+ 8C_R^{(\ast)}T\big(B_R^2 + 3K_R\big)\,h \Big)\, e^{2C_R^{(\ast)}\,T}
		\ +\, C_{T,p} R^{2-p}\\
		&= \Big( 5 T^2 \big[|b(0,\delta_0)| + 2\tilde L(1+R^\ell)R\big]^4 h^{2\alpha}
		+ 8C_R^{(\ast)}T\big(B_R^2 + 3K_R\big)\,h \Big)\, e^{2C_R^{(\ast)}\,T}
		\ +\, C_{T,p} R^{2-p}\\
		&=: (g_1+g_2)\,e^{2C_R^{(\ast)}T} \;+\; g_3.
	\end{aligned}
\end{equation}
Due to \eqref{oo} and \eqref{mmm}, we have
\begin{equation}\label{vvvv}
	C_R^{(\ast)} \le \bar C_1 \bigl(1 + R^{2\ell}\bigr),
	\qquad
	B_R^2 + 3K_R \le \bar C_2 \bigl(1 + R^{2\ell+2}\bigr),
\end{equation}
for some constants $\bar C_1,\bar C_2>0$ independent of $h$.  
Let

\begin{equation} \label{eq:R_choice}
	R = \left( \frac{\varepsilon}{4\bar{C}_1 T}\,\log\!\frac{1}{h} \right)^{\!1/(2\ell)}, \qquad 0<\varepsilon<1. 
\end{equation}
Hence, we have
\begin{equation}\label{ff}
	e^{2 C_R^{(\ast)} T}
	\le e^{\,2 \bar{C}_1 T \,(1 + R^{2\ell})}
	= e^{2 \bar{C}_1 T}\,\exp\!\Big(2 \bar{C}_1 T\cdot \frac{\varepsilon}{4\bar{C}_1 T}\,\log\!\tfrac{1}{h}\Big)
	= e^{2 \bar{C}_1 T}\,h^{-\varepsilon/2}.
\end{equation}
We now bound each term in \eqref{x} under the choice \eqref{eq:R_choice}:

\noindent\textbf{Step 1: Estimate of $g_1$.}
Since \eqref{eq:R_choice}, we have
\begin{equation}\label{zz}
	5 T^2 \left[ |b(0,\delta_0)| + 2\tilde{L}(1 + R^\ell) R \right]^4 h^{2\alpha}
	\le C_1\, h^{2\alpha}\, \bigl(\log(1/h)\bigr)^{\frac{2(\ell+1)}{\ell}} .
\end{equation}

\noindent\textbf{Step 2: Estimate of $g_2$.}
Due to \eqref{vvvv}, we get
\[
C_R^{(\ast)}(B_R^2 + 3K_R)
\le \bar{C}_1 (1 + R^{2\ell}) \cdot \bar{C}_2 (1 + R^{2\ell+2})
= \bar{C}_1 \bar{C}_2 \bigl(1 + R^{2\ell} + R^{2\ell+2} + R^{4\ell+2} \bigr).
\]
Hence,
\begin{equation}\label{xx}
	C_R^{(\ast)}(B_R^2 + 3K_R)
	\le C_2\, \bigl(\log(1/h)\bigr)^{\frac{4\ell + 2}{2\ell}} .
\end{equation}
Therefore, since in \eqref{x} we have $g_2= 8\,C_R^{(\ast)}T\,(B_R^2+3K_R)\,h$, combining with \eqref{xx} yields
	\(
	g_2 \le 8T\,C_2\, h\, (\log(1/h))^{\frac{4\ell+2}{2\ell}}.
	\)

\noindent\textbf{Step 3: Estimate of $g_3$.}
\begin{equation}\label{dd}
	C_{T,p} R^{2 - p}
	= C_{T,p}\, \bigl(\log(1/h)\bigr)^{\frac{2 - p}{2\ell}} .
\end{equation}
Substituing \eqref{ff}, \eqref{zz}, \eqref{xx}, \eqref{dd} into \eqref{x}, we obtain
\begin{align*}
	\mathbb{E} \left[ \sup_{0 \le t \le T} |e(t)|^2 \right]
	&\le \Big[ C_1\, h^{2\alpha} \bigl(\log(1/h)\bigr)^{\frac{2(\ell + 1)}{\ell}}
	\ +\, 8T\,C_2\, h\, \bigl(\log(1/h)\bigr)^{\frac{4\ell + 2}{2\ell}} \Big] \cdot h^{-\varepsilon/2} 
	\ +\, C_{T,p}\, \bigl(\log(1/h)\bigr)^{\frac{2 - p}{2\ell}}\\
	&= C_1\, h^{2\alpha-\frac{\varepsilon}{2}}\,\bigl(\log(1/h)\bigr)^{\frac{2(\ell + 1)}{\ell}}
	\ +\, 8T\,C_2\, h^{1-\frac{\varepsilon}{2}}\,\bigl(\log(1/h)\bigr)^{\frac{4\ell + 2}{2\ell}}
	\ +\, C_{T,p}\, \bigl(\log(1/h)\bigr)^{-\frac{p-2}{2\ell}}.
\end{align*}
Choosing \(\alpha\in\bigl(\tfrac12-\tfrac{\varepsilon}{4},\,\tfrac12\bigr)\) gives \(2\alpha-\tfrac{\varepsilon}{2}\ge 1-\varepsilon\). Hence,
\[
\mathbb{E} \left[ \sup_{0 \le t \le T} |e(t)|^2 \right]
\le C_{T,\varepsilon}\, h^{1-\varepsilon}\,(\log(1/h))^{C}
\ +\, C_{T,p}\, \bigl(\log(1/h)\bigr)^{-\frac{p-2}{2\ell}},
\]
for a constant \(C>0\) depending only on \(\ell\).
Consequently, the above bound guarantees strong convergence,
\[
\lim_{h\to 0}\mathbb{E} \left[ \sup_{0 \le t \le T} |e(t)|^2 \right] = 0.
\]

\end{proof}

\section{Numerical Examples}
In this section we present four conditional McKean-Vlasov SDE test problems with
common noise and jumps, and examine the performance of the TAEM scheme. All model
coefficients satisfy \textbf{(A1)-(A7)}. For a final time $T$ and a family of dyadic time grids
indexed by $l\in\{1,\ldots,6\}$ with time step $h_l=T/2^l$, we define the level-$l$ mean-squared error
(MSE) against the finest reference level $L_{\max}$ under a fully coupled driver construction:
\[
\mathrm{MSE}(l,T)
:= \frac{1}{M}\sum_{m=1}^{M} \big| X^{(l)}_m(T) - X^{(L_{\max})}_m(T)\big|^2,
\qquad
\beta := -\mathrm{Slope}\big(\, l \mapsto \log_2 \mathrm{MSE}(l,T) \,\big).
\]
Hence $\beta\approx 1$ corresponds to an almost first-order decay of the MSE with respect to $h$
(since $\mathrm{MSE}\approx O(h)$ implies a root mean-square error of order $O(h^{1/2})$).
We report the fitted $\beta$ for $T\in\{1,5,10\}$ with each model and show the aggregate plot
in Figure~\ref{fig:NumAll}.

\paragraph{Model 1: Ginzburg-Landau model with common noise.}
\[
{\rm d}X_t
= 0.1\bigl(X_t - X_t^{3} + m(\mu_t)\bigr)\,{\rm d}t
+ 0.1 X_t\,{\rm d}W_t
+ 0.05 X_t\,{\rm d}W_t^{0}
+ \sin(X_{t-})\,{\rm d}Z_t^{0},
\qquad X_0=1.0,
\]
where $\mu_t=\mathcal{L}^1(X_t)$ and $m(\mu_t)=\int_{\mathbb{R}} y\,\mu_t({\rm d}y)$.
Here $W$ denotes an idiosyncratic Brownian motion, $W^{0}$ a common Brownian motion,
and $Z^{0}$ a pure-jump L\'evy process (e.g., tempered-stable).
Due to Example~\ref{vv}, assumptions \textbf{(A1)-(A7)} are satisfied.

For Model~1, Figure~\ref{fig:NumAll} shows that the TAEM scheme exhibits an almost first-order decay
of the MSE at short horizon, with $\beta \approx 1.15$ for $T=1$. At $T=5$, we observe a
pre-asymptotic superlinear behaviour with $\beta \approx 1.78$. At $T=10$, the fitted rate turns
negative ($\beta \approx -0.48$) and the MSE increases upon refinement, reflecting long-horizon
instability driven by strong nonlinearity under common noise.

\paragraph{Model 2: LQ conditional mean-field control I model.}
\[
{\rm d}X_t = u(t)\,{\rm d}t + 0.15\,{\rm d}B_t^{0} 
+ 0.12\,\mathbb{E}[X_t \mid \mathcal{F}_t^{B^0}]\,{\rm d}W_t
+ \int_{\mathbb{R}} 0.08 \tanh(z)\,\tilde{N}({\rm d}t,{\rm d}z), 
\quad X_0=1.0,
\]
with control
\[
u(t) = -0.5\left(X_t + \mathbb{E}[X_t \mid \mathcal{F}_t^{B^0}]\right).
\]
Here $B^0$ denotes the common Brownian motion, $W$ an idiosyncratic Brownian motion, 
and $\tilde{N}$ a compensated Poisson random measure. 
The jump distribution is Gaussian with variance $\varsigma^2=0.25$, and the jump intensity is $\lambda=1.0$. 
Under these specifications, assumptions \textbf{(A1)-(A7)} are satisfied.

For Model~2, Figure~\ref{fig:NumAll} shows that we obtain 
$\beta\approx 1.01,\,1.59,\,1.91$ for $T=1,5,10$, respectively.
Accuracy remains robust and even improves within our tested range, reflecting that the control
effectively neutralizes both idiosyncratic and common fluctuations.

\paragraph{Model 3: LQ conditional mean-field control II model.}
\[
{\rm d}X_t = u(t)\,{\rm d}t 
+ \int_{\mathbb{R}} 0.18 \tanh(z)\,\tilde{N}_1({\rm d}t,{\rm d}z)
+ \int_{\mathbb{R}} 0.10 \tanh(z)\,\mathbb{E}[X_t \mid \mathcal{F}_t^{\tilde{N}_1}]\,\tilde{N}_2({\rm d}t,{\rm d}z), 
\quad X_0=1.0,
\]
with control
\[
u(t) = -0.4\left(X_t + \mathbb{E}[X_t \mid \mathcal{F}_t^{\tilde{N}_1}]\right).
\]
Here $\tilde{N}_1$ and $\tilde{N}_2$ are independent compensated Poisson random measures, 
representing idiosyncratic and common jumps, respectively. 
Both jump distributions are Gaussian with variance $0.16$, with intensities $\lambda_1=0.8$ and $\lambda_2=0.6$. 
By the Lipschitz and linear growth conditions, \textbf{(A1)-(A7)} are fulfilled. 

For Model~3, Figure~\ref{fig:NumAll} shows that the fitted rates are 
$\beta\approx 1.04,\,1.33,\,1.79$ for $T=1,5,10$, respectively,
again with slopes around or above $1$ and stable behaviour as $T$ grows.

\paragraph{Model 4: Interbank lending model with jumps.}
\[
{\rm d}X_t = 0.6\bigl(\bar{X}_t - X_t\bigr)\,{\rm d}t 
+ 0.25\,{\rm d}W_t + 0.18\,{\rm d}B_t^0 
+ \int_{\mathbb{R}} 0.12 \tanh(z)\,\tilde{N}({\rm d}t,{\rm d}z), 
\quad X_0=1.0,
\]
where $\bar{X}_t=\mathbb{E}[X_t\mid \mathcal{F}_t^{B^0}]$ denotes the conditional mean with respect to the common noise. 
This model captures mean-reversion toward the conditional average across banks, together with both idiosyncratic and common volatility. 
The jump distribution is Gaussian with variance $0.2025$, and the jump intensity is $\lambda=0.9$. 
Again, \textbf{(A1)-(A7)} are satisfied. 

For Model~4, Figure~\ref{fig:NumAll} shows that we obtain 
$\beta\approx 1.38,\,1.64,\,1.77$ for $T=1,5,10$, respectively,
with consistently large slopes and good large-$T$ behaviour; mean reversion stabilizes the
dynamics against the combined effect of common Brownian noise and Gaussian jumps.

\begin{figure}[H]
	\centering
	\includegraphics[width=0.9\textwidth]{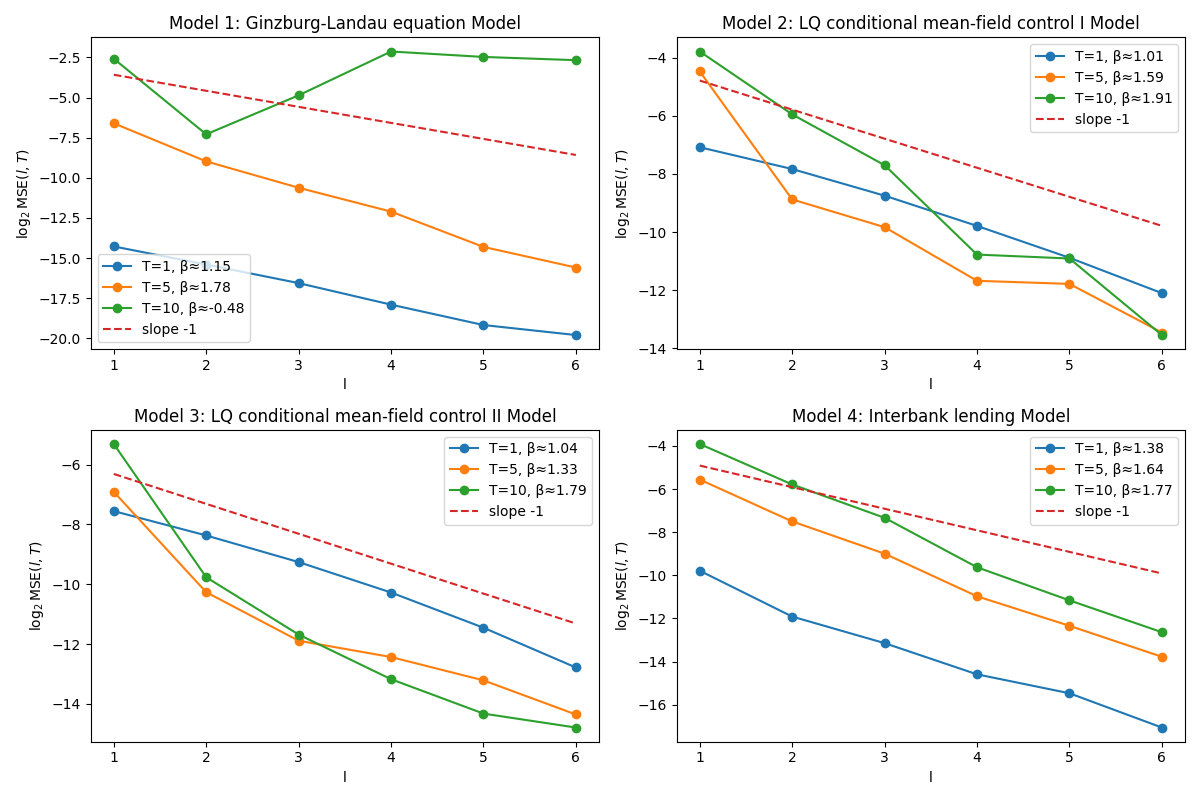}
	\caption{$\log_2\mathrm{MSE}(l,T)$ against $l$ for the four models ($T\!\in\!\{1,5,10\}$);
		dashed line: slope $-1$.}
	\label{fig:NumAll}
\end{figure}

\paragraph{Conclusion.}
In summary, the numerical experiments display an almost first-order decay of the MSE
in most cases (with fitted slopes $\beta \approx 1$), which complements the strong convergence
result in Theorem~\ref{cc}. Its long-horizon robustness hinges on the balance among
nonlinearity strength, common-noise coupling, and jump activity. Across the four models,
LQ-I/II and Interbank remain stable over extended horizons, while the Ginzburg-Landau model
shows a transient non-asymptotic behaviour only at very large $T$ due to strong nonlinearity.
Nevertheless, across practically relevant horizons all models demonstrate satisfactory
convergence behaviour and accuracy at the tested resolutions.

Moreover, the inclusion of common noise in all four models is essential. 
In practical financial and economic systems, individual agents are not only exposed to their own idiosyncratic shocks, 
but also to systemic sources of randomness such as market-wide fluctuations, macroeconomic policy shifts, or collective risk factors. 
This type of shared uncertainty governs the correlation structure among agents and strongly affects the evolution of the mean field 
as well as the effectiveness of control strategies. 
Neglecting common noise would lead to an overestimation of convergence and stability, and would fail to capture the mechanism of 
error accumulation over long horizons. 
Hence, incorporating common noise is crucial for accurately describing systemic risk propagation, 
assessing the robustness of numerical schemes, and ensuring that the models remain relevant for real-world applications.

	\section*{Acknowledgments}
	We would like to express our gratitude to all those who helped us with this research.
	Special thanks to Professor Jianhai Bao and Professor Jian Wang.

\end{document}